\documentclass[UTF-8,reqno]{amsart}
\usepackage{amsfonts,amsmath,amsthm,amssymb,stmaryrd}
\usepackage{hyperref}
\usepackage{color}
\usepackage{mathrsfs}

\newtheorem{theorem}{Theorem}[section]
\newtheorem{lemma}{Lemma}[section]
\newtheorem{proposition}{Proposition}[section]

\theoremstyle{definition}

\theoremstyle{remark}

\numberwithin{equation}{section}
\allowdisplaybreaks 
\begin{document}
\author{Chenlu Zhang}
\address{College of Mathematics and Statistics, Chongqing University,
                             Chongqing, 401331,  China.}
\email{20220601002@stu.cqu.edu.cn}
\author{Huaqiao Wang}
\address{College of Mathematics and Statistics, Chongqing University,
                             Chongqing, 401331,  China.}
\email{wanghuaqiao@cqu.edu.cn}

\title[Strong solutions of the LLS equation]
{Global existence of strong solutions to the Landau--Lifshitz--Slonczewski equation}
\thanks{Corresponding author: wanghuaqiao@cqu.edu.cn}
\
\begin{abstract}
In this paper, we focus on the existence of strong solutions for the Cauchy problem of the three-dimensional Landau-Lifshitz-Slonczewski
equation. We construct a new combination of Bourgain space and Lebesgue space where linear and nonlinear estimates can be closed by applying
frequency decomposition and energy methods. Finally, we establish the existence and uniqueness of the global strong solution provided that
the initial data belongs to Besov space $\dot{B}^{\frac{n}{2}}_{\Omega}$.
\end{abstract}

\maketitle
\section{Introduction}\label{Sec1}
The discovery of the spin-transfer torque (STT for short) effect is a milestone in the theory of micromagnetics and it breaks the traditional way of manipulating magnetic torque by magnetic fields. In the 1970s, Berger \cite{BerL,FreBerL,Hung} discovered that electric currents can drive the motion of magnetic domain walls. In the late 1980s, Slonczewski \cite{SlonJC} showed the existence of interlayer exchange coupling between the two ferromagnetic electrodes of a magnetic tunnel junction. Thereafter, the existence of STT effect was be confirmed, for details see \cite{Kub,San,Tso}. There are many types of STT, including Slonczewski STT, vertical STT, adiabatic and non-adiabatic STT. In the following, we focus on the third type of STT.

When the current flows along the film surface in any direction, the adiabatic and non-adiabatic STT effect are expressed by (see \cite{Smi}):
\begin{align*}
\boldsymbol{T}_{ad}&=\theta_{3}\boldsymbol{m}\times\left[\boldsymbol{m}\times(v\cdot\nabla\boldsymbol{m})\right],\\
\boldsymbol{T}_{na}&=\theta_{4}\boldsymbol{m}\times(v\cdot\nabla\boldsymbol{m}),
\end{align*}
where $v$ is the density of the spin-polarized current in the current direction, $\theta_{3}$ and $\theta_{4}$ are
dimensionless constants. $\boldsymbol{T}_{ad}$ denotes the adiabatic STT, which was first introduced by Bazaliy in \cite{Ral}.
$\boldsymbol{T}_{na}$ denotes the non-adiabatic STT, which was proposed by Zhang in \cite{Smi}.

If we only consider the effect of magnetic field, the dynamic behavior of magnetization can be described by the Landau-Lifshitz-Gilbert (LLG for short) equation:
\begin{align}
\frac{\partial{\boldsymbol{m}}}{\partial{t}}=\theta_{1}\boldsymbol{m}\times\boldsymbol{H}_{\rm eff}+\theta_{2}
\left(\boldsymbol{m}\times\frac{\partial{\boldsymbol{m}}}{\partial{t}}\right),\;\; \boldsymbol{m}(0,x)=\boldsymbol{m}_0.\label{1.8}
\end{align}
where $\boldsymbol{m}(t,x):\mathbb{R}\times\mathbb{R}^{n}\rightarrow\mathbb{S}^{2}\subset \mathbb{R}^{3}$ is the magnetic intensity, $\boldsymbol{H}_{\rm eff}$ is the effective field, $\theta_{1}$ is a constant related to the magnetogyric ratio, $\theta_{2}$ is the Gilbert damping parameter. The first term on the right-hand side of \eqref{1.8} represents Larmor precession and the second term is Gilbert damping term.

For the magnetic nanowires \cite{Hei,Mye} with in-plane current flow, the current can induce STT effect. In general, physicists \cite{Han} directly add the STT term to the right-hand side of \eqref{1.8}, then the LLG equation with STT effect is expressed by:
\begin{align}
\frac{{\partial} \boldsymbol{m}}{\partial t}&=\theta_{1}(\boldsymbol{m}\times\boldsymbol{H}_{\rm eff})
+\theta_{2}\left(\boldsymbol{m}\times\frac{{\partial}\boldsymbol{m}}{\partial t}\right)+\boldsymbol{T}_{ad}+\boldsymbol{T}_{na},\;\; \boldsymbol{m}(0,x)=\boldsymbol{m}_0. \label{1.12}
\end{align}

In this paper, we only consider the exchange field, i.e., $\boldsymbol{H}_{\rm eff}=\Delta\boldsymbol{m}$. By using the fact that $|\boldsymbol{m}|=1$, adding $\theta_{2}\boldsymbol{m}\times \eqref{1.12}$ into \eqref{1.12}, and omitting some coefficients for simplicity, we obtain the Landau-Lifshitz-Slonczewski (LLS for short) equation:
\begin{align}\label{main-eq}
\begin{cases}
{\partial_t}\boldsymbol{m}+(v\cdot\nabla)\boldsymbol{m}
+\boldsymbol{m}\times(v\cdot\nabla)\boldsymbol{m}
=\boldsymbol{m}\times\Delta\boldsymbol{m}
-\varepsilon(\boldsymbol{m}\times\boldsymbol{m}\times\Delta\boldsymbol{m}),\\
\boldsymbol{m}(0,x)=\boldsymbol{m}_{0},
\end{cases}
\end{align}
where $\varepsilon\in(0,1)$ is the Gilbert damping coefficient.

If we discard the STT effect, the LLS equation becomes the LLG equation. For the case of dimension $n=1$, Zhou-Guo \cite{ZhouG}
obtained the global existence of weak solutions for the LLG equation using the Leray-Schauder fixed point theorem. Later, Zhou-Guo-Tan \cite{ZhouGT} established the global existence and uniqueness of smooth solutions to the Cauchy problem with the periodic boundary by the difference method and the energy method. Afterwards, for the nonhomogeneous LLG equation, Ding-Guo-Su \cite{DingGS} considered the existence and uniqueness of local smooth solutions to the Cauchy problem with periodic boundary. In dimension $n=2$, Harpe \cite{Har} obtained the regularity of weak solutions to the initial-boundary value problem in a bounded open domain by using the Ginzburg-Landau approximation method. Furthermore, if the first-order derivative of the initial data is sufficiently small, Carbou-Fabrie \cite{Car} obtained the global existence of regular solutions of the initial-boundary value problem. When weak solutions have finite energy, Chen-Ding-Guo \cite{Chen} showed that the weak solution is regular and unique except for at most finite points in a compact two-dimensional manifold without boundary.

In higher dimensions, there have been many results for the LLG equation. Guo-Hong \cite{GuoHo} established the connection between the LLG equation and the harmonic map heat flow equation, and then they showed the existence of global weak solutions for the Cauchy problem by using penalized approximation method and also obtained the weak solution is regular except for at most finite points. Moreover, they considered the global smoothness of solutions when the regularity of the initial data is well enough or the initial energy is sufficiently small. Ding-Guo \cite{Ding} built up a special energy inequality and monotonicity inequality of weak solutions, and then showed that the weak solution is regular in a bounded compact Riemannian manifold. If the initial energy is sufficiently small, Ding-Wang \cite{DingW} proved that short-time smooth solutions must blow-up in finite time by using approximation method in an unbounded Riemannian manifold. In recent years, many researchers have noticed that the relationship between the LLG equation and the Ginzburg-Landau equation. Melcher \cite{Mel} transformed the LLG equation into the complex Ginzburg-Landau equation by moving frames and established the global existence and uniqueness of smooth solutions to the Cauchy problem in Sobolev space. Later, Lin-Lai-Wang \cite{Lin} also proved the global solvability of the Cauchy Problem in Morrey space. Recently, Guo-Huang \cite{GuoHu} investigated the existence and uniqueness of strong solutions in the critical Besov space by using the stereographic projection and frequency decomposition.

Obviously it is known that the well-posedness of the LLG equation have been established well, however there are a few results about the well-posedness of the LLS equation. In Sobolev space, Melcher-Ptashnyk \cite{MelP} studied the existence and uniqueness of global weak solutions and also considered smooth solutions when the initial data is small and regular sufficiently for the three-dimensional LLS equation. The regularity of the solution in Sobolev space seems very challenging for lack of regularity of the initial data, which also reflects the difficulties we encounter always in Sobobev space. Our goal is to investigate strong solutions  of the LLS equation with lower regularity of the initial data in other suitable space. The Bourgain space was first used to systematically study the low regularity theory of the Schr\"{o}dinger equation and the KdV equation by Bourgain \cite{Bour1,Bour2}. Later, the method has been widely applied to the KdV equation, nonlinear wave equations and the Schr\"{o}dinger equation, etc., for example see \cite{Ken,Kla,Tao}. The advantages of the Bourgain space are that we can lower the regularity of the initial data and use the Littlewood-Paley decomposition to transform the differential operator into scalar multiplication in frequency space, namely we use the frequency decomposition to balance the coefficients instead of reducing the regularity of space. By Duhamel's principle, the operator \;$e^{it\Delta}$\; appears in the expression of solutions and we cannot find a suitable estimate to deal with it in Bourgain space only. So we look for the semigroup estimates of the operator in the anisotropic Lebesgue space. Inspired by Guo-Huang \cite{GuoHu} who proved the existence of global strong solutions of the LLG equation in the critical Besov space, we apply the frequency decomposition to address the nonlinear terms of the LLS equation. Unlike the main nonlinear part of LLG equation is only $u(\nabla u)^{2}$, the major nonlinear parts of the LLS equation are $u(\nabla u)^{2}$ and $u^{2}\nabla u$. Since the sum of partial derivatives of the two terms are not equal, if we choose the same form of frequency decomposition, either we obtain different solution spaces or the coefficients in front of the same solution space are different. Both of these will construct a restricted and smaller solution space. Therefore, on the basis of the solution space previously obtained by the term $u(\nabla u)^{2}$, we exploit a new frequency decomposition to deal with the term $u^{2}\nabla u$ and combine a suitable solution space to close the energy of linear and nonlinear parts.

First, we transform the LLS equation \eqref{main-eq} into the complex derivative Ginzburg-Landau type equation by using the stereographic projection transform. Let
\begin{align*}
u=\frac{\boldsymbol{m}_{1}+i\boldsymbol{m}_{2}}{1+\boldsymbol{m}_{3}},
\end{align*}
where $\boldsymbol{m}=(\boldsymbol{m}_{1},\boldsymbol{m}_{2},\boldsymbol{m}_{3})$ is the solution of equation \eqref{main-eq}.
Obviously, the inverse of this projection is
\begin{align}\label{Ucoord}
(\boldsymbol{m}_{1},\boldsymbol{m}_{2},\boldsymbol{m}_{3})
=(\frac{u+\bar{u}}{1+|u|^{2}},\frac{-i(u-\bar{u})}{1+|u|^{2}},\frac{1-|u|^{2}}{1+|u|^{2}}),
\end{align}
Substituting \eqref{Ucoord} into \eqref{main-eq}, we obtain three equations respectively:
\begin{align}
&(1-\bar{u}^{2})A(u,\bar{u})+F(u,\bar{u})=(1-u^{2})\bar{A}(u,\bar{u})+\bar{F}(u,\bar{u}),\label{m1eq}\\
&(1+\bar{u}^{2})A(u,\bar{u})+H(u,\bar{u})=-(1+u^{2})\bar{A}(u,\bar{u})-\bar{H}(u,\bar{u}),\label{m2eq}\\
&\bar{u}A(u,\bar{u})=u\bar{A}(u,\bar{u}),\label{m3eq}
\end{align}
where
\begin{align*}
A(u,\bar{u})=iu_{t}-\left[(-1+\varepsilon i)\Delta u-(-1+i)(v\cdot\nabla)u-\frac{2(-1+\varepsilon i)
\bar{u}(\nabla u)^{2}}{1+|u|^{2}}\right],
\end{align*}
\begin{align*}
F(u,\bar{u})=\frac{2(1+\bar{u}^{2})(1-|u|^2)(v\cdot\nabla)\bar{u}}{1+|u|^{2}},\quad
H(u,\bar{u})=\frac{4(v\cdot\nabla)u(\bar{u}^{2}+|u|^{2})}{1+|u|^{2}}.
\end{align*}
Combining \eqref{m1eq} and \eqref{m2eq}, we have
\begin{align*}
A(u,\bar{u})=-u^{2}\bar{A}(u,\bar{u})+\frac{\bar{F}(u,\bar{u})-F(u,\bar{u})}{2}-\frac{H(u,\bar{u})+\bar{H}(u,\bar{u})}{2}.
\end{align*}
By \eqref{m3eq}, one has
\begin{align*}
A(u,\bar{u})=\frac{\left[\bar{F}(u,\bar{u}-F(u,\bar{u})\right]-\left[H(u,\bar{u})+\bar{H}(u,\bar{u})\right]}{2(1+|u|^{2})}
=\frac{{-i\rm{Im}}F-{\rm{Re}}H}{1+|u|^{2}},
\end{align*}
then $u$ solves the following complex derivative Ginzburg-Landau equation:
\begin{align*}
\begin{cases}
\partial_{t}u=-(\varepsilon+i)\Delta u+(1+i)(v\cdot\nabla)u+\frac{2(\varepsilon+i)\bar{u}(\nabla u)^{2}}{1+|u|^{2}}
+\frac{{\rm{Im}}F-i{\rm{Re}}H}{1+|u|^{2}},\\
u(0,x)=u_{0}.
\end{cases}
\end{align*}
Setting $\tilde{t}=-t$, we obtain
\begin{align}\label{deri-Ginz}
\begin{cases}
\partial_{t}u=(\varepsilon+i)\Delta u-(1+i)(v\cdot\nabla)u-\frac{2(\varepsilon+i)\bar{u}(\nabla u)^{2}}{1+|u|^{2}}
-\frac{{\rm{Im}}F-i{\rm{Re}}H}{1+|u|^{2}},\\
u(0,x)=u_{0}.
\end{cases}
\end{align}

Now, we state our main result in the following.
\begin{theorem}\label{main-th}
Assume that $n \geq 3$, $\|v\|_{L^{\infty}_{t,x}}\leq \eta$ and $\eta>0$ is small enough. When the initial data $\boldsymbol{m}_{0}=(\boldsymbol{m}_{01},\boldsymbol{m}_{02},\boldsymbol{m}_{03}) \in \dot{B}^{\frac{n}{2}}_{\Omega}(\mathbb{R}^{n}; \mathbb{S}^{2})$ satisfies $\left\|\frac{\boldsymbol{m}_{01}+i\boldsymbol{m}_{02}}{1+\boldsymbol{m}_{03}}\right\|_{F^{\frac{n}{2}}\cap Z^{\frac{n}{2}}}\leq \eta$, then the LLS equation \eqref{main-eq} has a unique strong solution $\boldsymbol{m}$ in the combination of Bourgain space and Lebesgue space, i.e., $F^{\frac{n}{2}}\cap Z^{\frac{n}{2}}$.
\end{theorem}

Observing that the connection between equations \eqref{main-eq} and \eqref{deri-Ginz}, we prove the well-posedness of equation \eqref{deri-Ginz} instead of equation \eqref{main-eq}. Our main strategy is to use the fixed point theorem; it is important to close the linear and nonlinear estimates. According to the structure of the solution, we construct a new solution space $F^{\frac{n}{2}}\cap Z^{\frac{n}{2}}$ for balancing the coefficients from the linear and nonlinear parts. More precisely, we first close the linear estimates in the space $F^{\frac{n}{2}}\cap Z^{\frac{n}{2}}$ (see Proposition \ref{pro linear-es}) by  Lemma \ref{Leb-linear}--Lemma \ref{X-linear}. Later, we prove that the nonlinear estimates in the space $F^{\frac{n}{2}}\cap Z^{\frac{n}{2}}$ is close (see Proposition \ref{pro nonlinear-es}) by using Lemma \ref{Lemma4.3}--Lemma \ref{Lemma4.2}. Combining the contractive result \eqref{contractive map} with Propositions \ref{pro linear-es} and \ref{pro nonlinear-es} which yield that the solution map is self-to-self, we obtain the well-posedness of equation \eqref{deri-Ginz}.

This paper is arranged as follows. In Section \ref{Sec2}, we introduce our main dyadic function spaces and recall some basic definitions.
In Section \ref{Sec3}, we list some useful lemmas and establish the linear estimate in Proposition \ref{pro linear-es}. In Section \ref{Sec4}, we obtain the nonlinear estimates in Proposition \ref{pro nonlinear-es}. In Section \ref{Sec5}, we devote to giving the proof of our main results (see Theorem \ref{main-th}).

\section{Definition and notation}\label{Sec2}
In this section, we recall some definitions and useful notations. For $\Omega \in \mathbb{S}^{2}$, the space $\dot{B}^{\frac{n}{2}}_{\Omega}$ is defined by
\begin{align*}
\dot{B}^{\frac{n}{2}}_{\Omega}:=\dot{B}^{\frac{n}{2}}_{\Omega}(\mathbb{R}^{n};\mathbb{S}^{2})=\{f:\mathbb{R}^{n}\rightarrow\mathbb{R}^{3};~
f-\Omega\in \dot{B}^{\frac{n}{2}}_{2,1},~|f(x)|=1~a.e.\text{ in }\mathbb{R}^{3}\},
\end{align*}
where the space $\dot{B}^{\frac{n}{2}}_{2,1}$ is the standard Besov space.

Let $\eta(|\xi|):\mathbb{R}\rightarrow[0,1]$ be a non-negative, smooth and radially decreasing function, where $0<\eta(|\xi|)\leq1$ compactly supported in $|\xi|\leq\frac{8}{5}$ and $\eta\equiv1$ when $|\xi|\leq\frac{5}{4}$.
Let $\chi_{k}(\xi)=\eta(\frac{|\xi|}{2^{k}})-\eta(\frac{|\xi|}{2^{k-1}})$,
$\chi_{\leq k}(\xi)=\eta(\frac{|\xi|}{2^{k}}),\tilde{\chi}_{k}(\xi)=\sum^{9n}_{l=-9n}\chi_{k+l}({\xi})$, we define the homogeneous and inhomogeneous Littlewood-Paley projector $P_{k}$ and $P_{\leq  k}$ on $L^{2}(\mathbb{R}^{n})$ respectively by
\begin{align*}
\widehat{P_{k}u}(\xi)=\chi_{k}(\xi)\widehat{u}(\xi),\quad \widehat{P_{\leq k}u}(\xi)=\chi_{\leq k}(\xi)\widehat{u}(\xi),
\end{align*}
for any $k\in\mathbb{Z}$, and $P_{\geq k}=I-P_{\leq k-1},\;P_{[k_{1},k_{2}]}=\sum_{j=k_{1}}^{k2}{P_{j}}$, where $I$ is the unit projection. Similarly, define $\widehat{\widetilde{P_{k}}u}=\widetilde{\chi_{k}}(\xi)\widehat{u}(\xi)$.

We define the modulation projector $Q_{k},Q_{\leq k}$ on $L^{2}(\mathbb{R}\times\mathbb{R}^{n})$ by
\begin{align*}
\widehat{Q_{k}u}(\xi,\tau)=\chi_{k}(\tau+|\xi|^{2})\widehat{u}(\xi,\tau),\quad
\widehat{Q_{\leq k}u}(\xi,\tau)=\chi_{\leq k}(\tau+|\xi|^{2})\widehat{u}(\xi,\tau),
\end{align*}
for any $k\in\mathbb{Z}$, where $Q_{\geq k}=I-Q_{\leq k-1},\;Q_{[k_{1},k_{2}]}=\sum_{j=k_{1}}^{k_{2}}{Q_{j}}$.

We define the anisotropic Lebesgue space $L^{p,q}_{e}(\mathbb{R}\times\mathbb{R}^{n}),~1\leq p,q<\infty$ by
\begin{align*}
\|f\|_{L^{p,q}_{\mathbf{e}}(\mathbb{R}^{n+1})}=\left(\int_{\mathbb{R}}\left(\int_{H_{\mathbf{e}}
\times\mathbb{R}}|f(\lambda \mathbf{e}+y,t)|^{q}dydt \right)^{\frac{p}{q}}d\lambda \right)^{\frac{1}{p}}.
\end{align*}
We decompose $\mathbb{R}^{n}=\lambda \mathbf{e}\oplus H_{\mathbf{e}}$, where $\mathbf{e}\in{\mathbb{S}^{n-1}}$ and $\mathbb{S}^{n-1}$
is the unit sphere on $\mathbb{R}^{n}$, $H_{\mathbf{e}}$ is the hyperplane with normal vector $\mathbf{e}$.
Write $L^{p,q}_{\mathbf{e}_{j}}=L^{p}_{x_{j}}L^{q}_{\bar{x_{j}},t}$, where $x=x_{i}\oplus\bar{x}_{i}$. The symbol $A\lesssim B$ means that $A\leq CB$ where $C$ is a positive  constant.

The Bourgain space $X^{s,b}$ has been used to the low regularity theory for the Cauchy problem of the nonlinear dispersive equation. In this paper, by means of the modulation-homogeneous version as in \cite{Beje,Guo}. We define $X^{0,b,q}$ by
\begin{align*}
\|f\|_{X^{0,b,q}}=\left(\sum_{k\in\mathbb{Z}}2^{kbq}\|Q_{k}f\|_{L^{2}_{t,x}}^{q}\right)^{1/q}.
\end{align*}

If$\;u(x,t)\in L^{2}(\mathbb{R}^{+}\times\mathbb{R}^{n})$ has spatial frequency in ${|\xi|\sim2^{k}}$, we define the main dyadic function space by
\begin{align*}
\|u\|_{F_{k}}=&\|u\|_{X^{0,\frac{1}{2},1}_{+}}+\|u\|_{L^{\infty}_{t}L^{2}_{x}}+\|u\|_{L^{2}_{t}L^{\frac{2n}{n-2}}_{x}}\\
&+2^{-(n-1)k/2}\sup_{\mathbf{e}_{i}\in\mathbb{S}^{n-1}}\|u\|_{L^{2,\infty}_{\mathbf{e}_{i}}}
+2^{\frac{k}{2}}\sup_{|j-k|\leq20}\sup_{\mathbf{e}_{i}\in\mathbb{S}^{n-1}}\|P_{j,\mathbf{e}_{i}}u\|_{L^{\infty,2}_{\mathbf{e}_{i}}},\\
\|u\|_{Y_{k}}=&\|u\|_{L^{\infty}_{t}L^{2}_{x}}+\|u\|_{L^{2}_{t}L^{\frac{2n}{n-2}}_{x}}
+2^{-(n-1)k/2}\sup_{\mathbf{e}_{i}\in\mathbb{S}^{n-1}}\|u\|_{L^{2,\infty}_{\mathbf{e}_{i}}}\\
&+2^{-k}\inf_{u=u_{1}+u_{2}}(\|u_{1}\|_{X^{0,1}}+\|u_{2}\|_{X^{0,1}}),\\
\|u\|_{Z_{k}}=&2^{-k}\|u\|_{X^{0,1}},\\
\|u\|_{N_{k}}=&\inf_{u=u_{1}+u_{2}+u_{3}}\left(\|u_{1}\|_{L^{1}_{t}L^{2}_{x}}
+2^{-k/2}\sup_{\mathbf{e}_{i}\in\mathbb{S}^{n-1}}\|u_{2}\|_{L^{1,2}_{\mathbf{e}_{i}}}
+\|u_{2}\|_{X^{0,-\frac{1}{2},1}}\right)+2^{-k}\|u\|_{L^{2}_{t,x}}.
\end{align*}
Obviously, $F_{k}\cap Z_{k}\subset Y_{k}$, then we define some spaces with the following norms:
\begin{align*}
\|u\|_{F^{s}}&=\sum_{k\in{\mathbb{Z}}}2^{ks}\|P_{k}u\|_{F_{k}},\quad \|u\|_{Y^{s}}=\sum_{k\in{\mathbb{Z}}}2^{ks}\|P_{k}u\|_{Y_{k}},\\
\|u\|_{Z^{s}}&=\sum_{k\in{\mathbb{Z}}}2^{ks}\|P_{k}u\|_{Z_{k}},\quad \|u\|_{N^{s}}=\sum_{k\in{\mathbb{Z}}}2^{ks}\|P_{k}u\|_{N_{k}}.
\end{align*}

\section{The linear estimate}\label{Sec3}
In order to prove that the solution map is closed, we need to discuss the property of solution spaces. In this section, we estimate the linear parts of equation \eqref{deri-Ginz}, and obtain the uniform estimates with respect to $\varepsilon$. Inspired by \cite{IonK,IonK1,Keel}, we can resort to some conclusions similar to Strichartz estimates from the Schr\"{o}dinger equation:
\begin{lemma}\label{semi-stri}
Assume that $n\geq3$, for any $k\in\mathbb{Z}$, we have
\begin{align*}
&\|e^{it\Delta}P_{k}f\|_{L^{2}_{t}L^{\frac{2n}{n-2}}_{x}\cap L^{\infty}_{t}L^{2}_{x}}
\leq C\|P_{k}f\|_{L^{2}_{x}}, \\
&\sup_{\mathbf{e}_{j}\in\mathbb{S}^{n-1}}\|e^{i t\Delta}P_{k}f\|_{L^{2,\infty}_{\mathbf{e}_{j}}}
\leq C2^{\frac{(n-1)k}{2}}\|P_{k}f\|_{L^{2}_{t,x}},\\
&\sup_{\mathbf{e}_{j}\in\mathbb{S}^{n-1}}\|e^{i t\Delta}P_{k,\mathbf{e}_{j}}f\|_{L^{\infty,2}_{\mathbf{e}_{j}}}
\leq C2^{-\frac{k}{2}}\|P_{k}f\|_{L^{2}_{t,x}}.
\end{align*}
where the constant $C>0$ is independent of $k$.
\end{lemma}

Notice that the space $F_{k}\cap Z_{k}$ consists of three main spaces: the general Lebesgue space $L_{t}^{p}L_{x}^{q}$, the anisotropic Lebesgue space $L_{\mathbf{e}}^{p,q}$ and $X^{0,b,q}$. We introduce a present lemma for connecting the norm of $X^{0,b,q}$ with other norms which is an extension of \cite[Proposition 5.4]{Wang}.

\begin{lemma}\label{semi-Banach}
Assume that $\mathbb{X}$ is an arbitrary space-time Banach space, if for any $f_{0}\in L^{2}_{x}$, $\tau_{0}\in \mathbb{R}$, one has
\begin{align*}
\|e^{it\tau_{0}}S(t)f_{0}\|_{\mathbb{X}}\leq C(k)\|f_{0}\|_{L^{2}_{x}},
\end{align*}
then for any $k\in \mathbb{Z}$, $f\in L^{2}_{t,x}$,
\begin{align*}
\|P_{k}f\|_{\mathbb{X}}\leq C(k)\|P_{k}f\|_{X^{0,\frac{1}{2},1}},
\end{align*}
where $C(k)$ is the polynomial with respect to $k$.
\end{lemma}

In the following we discuss the linear estimates of $u(t,x)$ in each of the three spaces.
\begin{lemma}[The general Lebesgue space]\label{Leb-linear}
Assume that $n\geq3$, $u(t,x)$ is the solution of \eqref{deri-Ginz}, $J(u(t,x))$ is the nonlinear term, for any $\varepsilon>0$,
\begin{align*}
u_{t}-(\varepsilon+i)\Delta u=J(u(t,x)),\quad u(0,x)=u_{0}.
\end{align*}
Then for any $\mathbf{e}_{i}\in \mathbb{S}^{n-1}$, we have
\begin{align}
\|P_{k}u\|_{L^{2}_{t}L^{\frac{2n}{n-2}}_{x}\cap L^{\infty}_{t}L^{2}_{x}}
\leq C\left(\|u_{0}\|_{L^{2}_{x}}+\|J(u(t,x))\|_{N_{k}}\right),\label{Leb-linear-eq}
\end{align}
where the constant $C>0$ is independent of $\varepsilon$ and $k$.
\end{lemma}
\begin{proof}
By the Duhamel's principle, we obtain
\begin{align*}
u=e^{(\varepsilon+ i)t\Delta}u_{0}+\int_{0}^{t}e^{(\varepsilon+ i)(t-s)\Delta}J(s)ds.
\end{align*}

Firstly, assuming that the term $J(u(t,x))=0$ and using the Fourier transform, one has
\begin{align*}
e^{\varepsilon t\Delta}(e^{i t\Delta}u_{0})
&=\mathscr{F}^{-1}_{\xi}\{e^{-\varepsilon t|\xi|^{2}}\cdot e^{-i t|\xi|^{2}}\hat{u_{0}}\}\\
&=\mathscr{F}^{-1}_{\xi}\{e^{-\varepsilon t|\xi|^{2}}\}\ast e^{i t\Delta}u_{0} \\
&\lesssim (\varepsilon t)^{-\frac{n}{2}}e^{-\frac{x^{2}}{4\varepsilon t}} \ast e^{i t\Delta}u_{0},
\end{align*}
then by the $p-p$ boundedness of the operator $P_{k}$, Young's inequality and Lemma \ref{semi-stri}, we have
\begin{align*}
\|P_{k}u\|_{L^{2}_{t}L^{\frac{2n}{n-2}}_{x}\cap L^{\infty}_{t}L^{2}_{x}}
&\lesssim \|e^{(\varepsilon+ i)t\Delta}u_{0}\|_{L^{2}_{t}L^{\frac{2n}{n-2}}_{x}\cap L^{\infty}_{t}L^{2}_{x}} \\
&\lesssim\|(\varepsilon t)^{-\frac{n}{2}}e^{-\frac{x^{2}}{4\varepsilon t}}\|_{L^{2}_{t}L^{1}_{x}\cap L^{\infty}_{t}L^{1}_{x}}
\|e^{i t\Delta}u_{0}\|_{L^{2}_{t}L^{\frac{2n}{n-2}}_{x}\cap L^{\infty}_{t}L^{2}_{x}}  \\
&\lesssim \|u_{0}\|_{L^{2}_{x}}.
\end{align*}

Secondly, we suppose that $u_{0}=0$. Notice that the norm $N^{\frac{n}{2}}$ consists of four parts, which are proved separately below. We extend the term $J(u(t,x))$ to the term $\tilde{J}(u(t,x))$ satisfying
\begin{align*}
\|\tilde{J}(u(t,x))\|_{X^{0,-\frac{1}{2},\infty}}\lesssim \|J(u(t,x))\|_{X^{0,-\frac{1}{2},\infty}_{+}}.
\end{align*}
and define $\tilde{u}=\mathscr{F}^{-1}_{\tau,\xi}\frac{1}{\tau+|\xi|^{2}+i\varepsilon|\xi|^{2}}\mathscr{F}_{t,x}\tilde{J}(u(t,x))$.

Next, we want to show
\begin{align*}
\|P_{k}u\|_{L^{2}_{t}L^{\frac{2n}{n-2}}_{x}\cap L^{\infty}_{t}L^{2}_{x}}
\lesssim \|J(u(t,x))\|_{X^{0,-\frac{1}{2},1}_{+}}.
\end{align*}
Using Lemma \ref{semi-Banach}, we get
\begin{align}
\begin{split}
\|P_{k}\tilde{u}(t,x)\|_{L^{2}_{t}L^{\frac{2n}{n-2}}_{x}\cap L^{\infty}_{t}L^{2}_{x}}
&\lesssim \|P_{k}\tilde{u}(t,x)\|_{X^{0,\frac{1}{2},1}} \\
&\lesssim \sum_{j\in\mathbb{Z}}2^{\frac{j}{2}}\left\|\mathscr{F}^{-1}_{\tau,\xi}\bigg\{\chi_{j}(\tau+|\xi|^{2})\chi_{k}(\xi)
\frac{1}{\tau+|\xi|^{2}+i\varepsilon|\xi|^{2}}\mathscr{F}_{t,x}\{\tilde{J}(u)\}\bigg\}\right\|_{L^{2}_{t,x}} \\
&\lesssim \sum_{j\in\mathbb{Z}}2^{\frac{j}{2}}\left\|\mathscr{F}^{-1}_{\tau,\xi}\bigg\{\chi_{j}(\tau+|\xi|^{2})
\mathscr{F}_{t,x}\{\tilde{J}(u)\}\bigg\}\right\|_{L^{2}_{t,x}}\cdot2^{-j} \\
&\lesssim \|J(u(t,x))\|_{X^{0,-\frac{1}{2},1}_{+}}. \label{3.2}
\end{split}
\end{align}
Inspired by \cite{Keel}, we can easily show that
\begin{align*}
\|P_{k}u\|_{L^{2}_{t}L^{\frac{2n}{n-2}}_{x}}\lesssim \|J(u(t,x))\|_{L^{1}_{t}L^{2}_{x}}.
\end{align*}
Notice that $(2,\frac{2n}{n-2})\in \Lambda_{0}$ (see \cite{Keel}), and $(q,r)=(2,\frac{2n}{n-2}), (q',r')=(1,2)$ satisfy
\begin{align}
\frac{1}{q}+\frac{n}{r}=\frac{1}{q'}+\frac{n}{r'}-2. \label{3.9}
\end{align}
By Theorem 2.1 in \cite[Introduction]{Keel}, we obtain
\begin{align}
&\left\|\int_{0}^{t}e^{i(t-s)\Delta} J(u(s))ds\right\|_{L^{2}_{t}L^{\frac{2n}{n-2}}_{x}}\lesssim \|J(u)\|_{L^{1}_{t}L^{2}_{x}},\label{3.9.1}\\
&\|e^{it\Delta} u(0)\|_{L^{\infty}_{t}L^{2}_{x}}\lesssim \|u(0)\|_{L^{2}},\label{3.9.2}\\
&\left\|\int_{-\infty}^{+\infty}e^{-is\Delta}J(u(s))ds\right\|_{L^{2}}\lesssim \|J(u)\|_{L^{1}_{t}L^{2}_{x}}.\label{3.9.3}
\end{align}
Combining \eqref{3.9.1} with  Young's inequality, one has
\begin{align*}
&\left\|\int_{0}^{t}e^{i(t-s)\Delta+\varepsilon(t-s)\Delta} J(u(s))ds\right\|_{L^{2}_{t}L^{\frac{2n}{n-2}}_{x}}\\
&\lesssim \left\|\int_{0}^{t}[\varepsilon(t-s)]^{-\frac{n}{2}}\|e^{-\frac{x^{2}}{4\varepsilon(t-s)}}\ast e^{i(t-s)\Delta}J(u(s))\|
_{L^{\frac{2n}{n-2}}_{x}}ds\right\|_{L^{2}_{t}}\\
&\lesssim \left\|\int_{0}^{t}e^{i(t-s)\Delta}J(u(s))ds\right\|_{L^{2}_{t}L^{\frac{2n}{n-2}}_{x}}
\lesssim \|J(u)\|_{L^{1}_{t}L^{2}_{x}}.
\end{align*}

Now we prove that
\begin{align*}
\|P_{k}u\|_{L^{\infty}_{t}L^{2}_{x}}\lesssim \|J(u(t,x))\|_{L^{1}_{t}L^{2}_{x}}.
\end{align*}
Unlike the previous proof, $(\infty,2), (1,2)$ are conjugate indices and do not satisfy the relationship of \eqref{3.9}. Considering \eqref{3.9.2} and \eqref{3.9.3}, employing the Plancherel equality and Bochner's inequality, we find that
\begin{align*}
\|P_{k}u\|_{L^{\infty}_{t}L^{2}_{x}}
&\lesssim \left\|\int_{0}^{t}\|e^{-i(t-s)|\xi|^{2}-\varepsilon(t-s)|\xi|^{2}}\hat{J}(u(s))\|_{L^{\infty}_{t}}ds\right\|_{L^{2}_{xi}}\\
&\lesssim \left\|\int_{-\infty}^{+\infty}e^{is|\xi|^{2}}\hat{J}(u(s))ds\right\|_{L^{2}_{\xi}}
\|e^{-it|\xi|^{2}-\varepsilon(t-s)|\xi|^{2}}\|_{L^{\infty}_{t,s,\xi}}\\
&\lesssim\|J(u)\|_{L^{1}_{t}L^{2}_{x}}.
\end{align*}

Finally, we show that
\begin{align*}
\|P_{k}u\|_{L^{2}_{t}L^{\frac{2n}{n-2}}_{x}\cap L^{\infty}_{t}L^{2}_{x}}
\lesssim 2^{-\frac{k}{2}}\|J(u(t,x))\|_{L^{1,2}_{\mathbf{e}_{j}}}.
\end{align*}
Similar to the argument of \eqref{3.12} in Lemma \ref{anisLeg-linear}, we can get the above estimate. Then, we complete the proof of \eqref{Leb-linear-eq}.
\end{proof}

\begin{lemma}[The anisotropic Lebesgue space]\label{anisLeg-linear}
Assume that $n\geq3$, $u(t,x)$ is the solution of \eqref{deri-Ginz}, $J(u(t,x))$ is the nonlinear term, for any $\varepsilon>0$,
\begin{align*}
u_{t}-(\varepsilon+i)\Delta u=J(u(t,x)),\quad u(0,x)=u_{0},
\end{align*}
for any $\mathbf{e}_{i}\in \mathbb{S}^{n-1}$, we have
\begin{align}
&\|P_{k}u\|_{L^{2,\infty}_{\mathbf{e}_{i}}}
\leq C\big( 2^{k(n-1)/2}\|u_{0}\|_{L^{2}_{x}}+2^{k(n-2)/2}\sup_{\mathbf{e}_{i}\in \mathbb{S}^{n-1}}\|J(u(t,x))\|_{L^{1,2}_{\mathbf{e}_{i}}}\big),\label{3.12}\\
&\|P_{k,\mathbf{e}_{i}}u\|_{L^{\infty,2}_{\mathbf{e}_{i}}}
\leq C\big( 2^{-k/2}\|u_{0}\|_{L^{2}_{x}}+2^{-k}\sup_{\mathbf{e}_{i}\in \mathbb{S}^{n-1}}
\|J(u(t,x))\|_{L^{1,2}_{\mathbf{e}_{i}}}\big),\label{3.13}
\end{align}
where the constant $C>0$ is independent of $\varepsilon$ and $k$.
\end{lemma}

\begin{proof}
First, we show the second inequality \eqref{3.13}. Assuming that the term $J(u(t,x))=0$, using the Plancherel equality, convolution theorem and Young's inequality in convolution form, we get
\begin{align*}
\|P_{k,\mathbf{e}_{i}}u\|_{L^{\infty,2}_{\mathbf{e}_{i}}}
&\lesssim \left\|\mathscr{F}^{-1}_{\xi_{i}}\big\{e^{-(i+\varepsilon)t\xi_{i}^{2}}\mathscr{F}_{x}u_{0}(x)\big\}\right\|_{L^{\infty}_{x_{i}}
L^{2}_{\bar{\xi}_{i}}L^{2}_{t}}\\
&=\left\|\mathscr{F}^{-1}_{\xi_{i}}\big\{\mathscr{F}_{x_{i}}\{e^{(i+\varepsilon)t\partial^{2}x_{i}}\}\cdot
\mathscr{F}_{x_{i}}\{\mathscr{F}_{\bar{x}_{i}}u_{0}(x)\}\big\}\right\|_{L^{\infty}_{x_{i}}L^{2}_{\bar{\xi}_{i}}L^{2}_{t}}\\
&\lesssim\|e^{(i+\varepsilon)t\partial^{2}x_{i}}\|_{L^{2}_{x_{i}}L^{2}_{t}}
\|\mathscr{F}_{\bar{x}_{i}}u_{0}(x)\|_{L^{2}_{x_{i},\bar{\xi}_{i}}}\\
&\lesssim 2^{-\frac{k}{2}}\|u_{0}(x)\|_{L^{2}_{x}}.
\end{align*}

Next we assume that the term $u_{0}=0$ (see \cite[Lemma 2.4]{HanW}). We divide the Fourier transform of space into $x_{i}$ and $\bar{x}_{i}$, use the Plancherel equality and the convolution theorem to get
\begin{align*}
\|P_{k,\mathbf{e}_{i}}u\|_{L^{\infty,2}_{\mathbf{e}_{i}}}
&\lesssim \left\|\mathscr{F}^{-1}_{\xi_{i}}\mathscr{F}_{t}\bigg\{\int_{-\infty}^{+\infty}
e^{-i(t-s)\xi_{i}^{2}-\varepsilon(t-s)\xi_{i}^{2}}\mathscr{F}_{x}J(u(s,x))ds\bigg\}\right\|_{L^{\infty}_{x_{i}}
L^{2}_{\bar{\xi}_{i}}L^{2}_{\tau}}\\
&\lesssim \left\|\mathscr{F}^{-1}_{\xi_{i}}\bigg\{\mathscr{F}_{t}\big\{e^{-it\xi_{i}^{2}-\varepsilon t\xi_{i}^{2}}\big\}
\cdot \mathscr{F}_{t}\big\{\mathscr{F}_{x}J(u(s,x))\big\}\bigg\}\right\|_{L^{\infty}_{x_{i}}L^{2}_{\bar{\xi}_{i}}L^{2}_{\tau}}\\
&\lesssim \left\|\mathscr{F}^{-1}_{\tau,\xi_{i}}\bigg\{\frac{1}{i(\tau+\xi_{i}^{2})+\varepsilon \xi_{i}^{2}}
 \mathscr{F}_{t}\big\{\mathscr{F}_{y}J(u(s,y))\big\}\bigg\}\right\|_{L^{\infty}_{x_{i}}L^{2}_{\bar{\xi}_{i}}L^{2}_{t}}\\
&\lesssim \left\|\int_{y_{i}}K(\tau,z_{i})\mathscr{F}_{t,\bar{y}_{i}}\big\{J(u(s,y_{i},\bar{y}_{i}))\big\}dy_{i}\right\|
_{L^{\infty}_{x_{i}}L^{2}_{\bar{\xi}_{i}}L^{2}_{\tau}}\\
&\lesssim \|K(\tau,z_{i})\|_{L^{\infty}_{\tau,z_{i}}} \left\|\int_{y_{i}}\mathscr{F}_{t,\bar{y}_{i}}
\big\{J(u(s,y_{i},\bar{y}_{i}))\big\}dy_{i} \right\|_{L^{2}_{\tau,\bar{\xi}_{i}}}\\
&\lesssim \|K(\tau,z_{i})\|_{L^{\infty}_{\tau,z_{i}}} \|J(u(t,x))\|_{L^{1,2}_{\mathbf{e}_{i}}},
\end{align*}
where $z_{i}=x_{i}-y_{i}$ and
\begin{align*}
K(\tau,z_{i}):=\int_{\xi_{i}}e^{iz_{i}\xi_{i}}\frac{1}{i(\tau+\xi_{i}^{2})+\varepsilon \xi_{i}^{2}}d\xi_{i}.
\end{align*}
Since the operator $P_{k,\mathbf{e}_{i}}$ is supported in $\xi_{i}\sim 2^{k+9n}$, we can get
\begin{align*}
\|K(\tau,z_{i})\|_{L^{\infty}_{\tau,z_{i}}}
\lesssim \sup_{\tau\neq0}\left|\int_{\xi_{i}}\frac{1}{|\tau+\xi_{i}^{2}|}d\xi_{i}\right|
+\left|\int_{\xi_{i}}\frac{1}{\xi_{i}^{2}}d\xi_{i}\right|
\lesssim 2^{-k}.
\end{align*}
Then, we obtain \eqref{3.13}.

Next, we prove the inequality \eqref{3.12}. Assuming that the term $J(u(t,x))=0$, using the $p-p$ boundedness of the operator $P_{k}$,  and combining Lemma \ref{semi-stri}, it holds
\begin{align*}
\|P_{k}u\|_{L^{2,\infty}_{\mathbf{e}_{i}}}
&\lesssim \|e^{(\varepsilon+ i)t\Delta}u_{0}\|_{L^{2,\infty}_{\mathbf{e}_{i}}}
\lesssim\|u_{0}\|_{L^{2}_{x}},
\end{align*}
where the last line is obtained in the similar way as Lemma \ref{Leb-linear}.
We assume that the term $u_{0}=0$, decompose $P_{k}u=\sum_{i=1}^{n}U_{i}$ such that $\mathscr{F}_{x}U_{i}$ is supported in $\{|\xi|\sim2^{k}: \xi_{i}\sim 2^{k}\}$ and decompose the term $u(t,x)$ as follows:
\begin{align*}
u(t,x)&=\int_{\mathbb{R}^{n+1}}\frac{e^{it\tau}e^{ix\xi}}{\tau+|\xi|^{2}+i\varepsilon |\xi|^{2}}\hat{J}(\tau,\xi)d\xi d\tau \\
&=\int_{\mathbb{R}^{n+1}}\frac{e^{it\tau}e^{ix\xi}}{\tau+|\xi|^{2}+i\varepsilon |\xi|^{2}}\hat{J}(\tau,\xi)d\xi d\tau
(1_{\{-\tau-|\bar{\xi}_{i}|^{2}\sim 2^{2k}\}^{c}}+1_{\{-\tau-|\bar{\xi}_{i}|^{2}\sim 2^{2k},\;|\tau+|\xi|^{2}|\lesssim\varepsilon 2^{2k}\}}\\
&\quad+1_{\{-\tau-|\bar{\xi}_{i}|^{2}\sim 2^{2k},\;|\tau+|\xi|^{2}|\gg\varepsilon 2^{2k}\}})d\xi d\tau \\
&=:u_{1}+u_{2}+u_{3}.
\end{align*}
For the term $u_{1}$, applying the Plancherel equality and the boundedness, we get
\begin{align*}
\|\Delta u_{1}\|_{L^{2}_{t,x}}
\lesssim \left\|\frac{|\xi|^{2}}{\tau+|\xi|^{2}+i\varepsilon|\xi|^{2}}\right\|_{L^{\infty}_{\tau,\xi}}\|J(t,x)\|_{L^{2}_{t,x}}
\lesssim \|J(t,x)\|_{L^{2}_{x}},\\
\|\partial_{t} u_{1}\|_{L^{2}_{t,x}}
\lesssim \left\|\frac{i\tau}{\tau+|\xi|^{2}+i\varepsilon|\xi|^{2}}\right\|_{L^{\infty}_{\tau,\xi}}\|J(t,x)\|_{L^{2}_{t,x}}
\lesssim \|J(t,x)\|_{L^{2}_{t,x}}.
\end{align*}
By the Sobolev embedding, we deduce that
\begin{align*}
\|u_{1}\|_{L^{2}_{x_{i}}L^{\infty}_{\bar{x}_{i}}}\lesssim \|u_{1}\|_{W^{2,2}_{x_{i}},W^{2,2}_{\bar{x}_{i}}}
\lesssim \|J(t,x)\|_{L^{2}_{x}},
\end{align*}
\begin{align*}
\|u_{1}\|_{L^{\infty}_{t}} \lesssim \|u_{1}\|_{W^{1,2}_{t}} \lesssim \|J(t,x)\|_{L^{2}_{t}},
\end{align*}
and then by applying Bernstein's inequality we obtain
\begin{align*}
\|u_{1}\|_{L^{2,\infty}_{\mathbf{e}_{i}}}\lesssim \|J(t,x)\|_{L^{2}_{t,x}} \lesssim2^{k/2}\|J(t,x)\|_{L^{1,2}_{\mathbf{e}_{i}}}
\lesssim2^{(n-2)k/2}\|J(t,x)\|_{L^{1,2}_{\mathbf{e}_{i}}}.
\end{align*}
For the term $u_{2}$, using Lemma \ref{semi-stri}, we get
\begin{align*}
\|e^{it\tau_{0}}e^{it\Delta}P_{k}u_{2}(0)\|_{L^{2,\infty}_{\mathbf{e}_{i}}}
\lesssim2^{(n-1)k/2}\|P_{k}u_{2}(0)\|_{L^{2}_{t,x}}.
\end{align*}
By Lemma \ref{semi-Banach} and the Plancherel equality, we have
\begin{align*}
&\|P_{k}u_{2}\|_{L^{2,\infty}_{\mathbf{e}_{i}}}\\
&\lesssim 2^{(n-1)k/2}\|P_{k}u_{2}\|_{X^{0,\frac{1}{2},1}}\\
&\lesssim 2^{(n-1)k/2}\sum_{j\leq \log_{2}{(\varepsilon 2^{2k})}}2^{\frac{j}{2}}
\left\|\int_{\mathbb{R}^{n+1}}\frac{e^{it\tau}e^{ix\xi}}{\tau+|\xi|^{2}+i\varepsilon |\xi|^{2}}\hat{J}(\tau,\xi)
1_{\{-\tau-|\overline{\xi_{i}}|^{2}\sim 2^{2k},\;|\tau+|\xi|^{2}|\lesssim\varepsilon 2^{2k}\}}d\xi d\tau\right\|_{L^{2}_{t,x}} \\
&\lesssim 2^{(n-1)k/2}\sum_{j\leq \log_{2}{(\varepsilon 2^{2k})}}2^{\frac{j}{2}} \left\|\frac{1}{\tau+|\xi|^{2}+i\varepsilon |\xi|^{2}}\right\|_{L^{\infty}_{\tau,\xi}}
\|J(t,x)\|_{L^{2}_{t,x}} \\
&\lesssim 2^{(n-1)k/2}\varepsilon^{\frac{1}{2}}2^{k}\cdot\varepsilon^{-1}2^{-2k}\|J(t,x)\|_{L^{2}_{t,x}}\\
&\lesssim \varepsilon^{-\frac{1}{2}}2^{\frac{(n-3)}{2}k}\|J(t,x)\|_{L^{2}_{t,x}}.
\end{align*}
Letting $\lambda:=\sqrt{-\tau-|\bar{\xi}_{i}|^{2}}$, observe that $\tau+|\xi|^{2}=-(\lambda-\xi_{i})(\lambda+\xi_{i})$, thus we get
$|\lambda+\xi_{i}|\sim 2^{k},\; |\lambda-\xi_{i}|\lesssim \varepsilon 2^{k}$, and $\xi_{i}\sim \varepsilon 2^{k}$. We use Bernstein's inequality with respect to $x_{i}$ to obtain the desired estimate:
\begin{align*}
\varepsilon^{-\frac{1}{2}}2^{\frac{(n-3)}{2}k}\|J(t,x)\|_{L^{2}_{t,x}}
\lesssim 2^{(n-2)k/2}\|J(t,x)\|_{L^{1,2}_{\mathbf{e}_{i}}}.
\end{align*}
For the term $u_{3}$, by Taylor's expansion: {\small
\begin{align*}
\frac{1}{\tau+|\xi|^{2}+i\varepsilon |\xi|^{2}}
&=\frac{1}{\tau+|\xi|^{2}}+\sum_{k=1}^{\infty}\frac{(-i\varepsilon|\xi|^{2})^{k}}{(2\lambda(\xi_{i}-\lambda))^{k+1}}
+\sum_{k=1}^{\infty}\frac{(-i\varepsilon|\xi|^{2})^{k}}{(2\lambda(\xi_{i}-\lambda))^{k+1}}
\left[\left(1+\frac{\lambda-\xi_{i}}{\lambda+\xi_{i}}\right)^{k+1}-1\right]\\
&=:\varphi_{1}(\tau,\xi)+\varphi_{2}(\tau,\xi)+\varphi_{3}(\tau,\xi),
\end{align*}}
one has
\begin{align*}
u_{3}^{j}=\int_{\mathbb{R}^{n+1}}e^{it\tau}e^{ix\xi}\varphi_{j}(\tau,\xi)\hat{J}(\tau,\xi)1_{\{-\tau-|\bar{\xi}_{i}|^{2}\sim 2^{2k},\;|\tau+|\xi|^{2}|\gg\varepsilon 2^{2k}\}} d\xi d\tau,\,\,j=1,2,3.
\end{align*}

For the term $u_{3}^{1}$, this corresponds to the case $\varepsilon=0$ which has been proved in \cite[Lemma 4.1]{IonK}. For the term $u_{3}^{2}$, we separate $\bar{\xi_{i}}$ and $\xi_{i}$ as much as possible,
\begin{align}
\begin{split}
u_{3}^{2}&=\sum_{k=1}^{\infty}(-i\varepsilon)^{k}\int_{\bar{\xi_{i}}\times\tau}e^{it\tau}e^{i\bar{x_{i}}\bar{\xi}_{i}}(2\lambda)^{-k-1}
1_{-\tau-|\bar{\xi}|^{2}\sim 2^{2k}}\\
&\quad\int_{\xi_{i}}\frac{e^{ix_{i}\xi_{i}}|\xi|^{2k}1_{|\tau+|\xi|^{2}|\gg\varepsilon2^{2k}}}{(\xi_{i}-\lambda)^{k+1}}
\hat{J}(\xi_{i},\bar{\xi}_{i},\tau)d\xi_{i}d\bar{\xi_{i}}d\tau.\label{3.17}
\end{split}
\end{align}
Let
\begin{align*}
K(y_{i},\bar{\xi_{i}},\tau)=\mathscr{F}^{-1}_{\xi_{i}}\bigg\{|\xi|^{2k}\hat{J}(\xi_{i},\bar{\xi}_{i},\tau)
1_{|\lambda+\xi_{i}|\sim2^{k}}\bigg\},
\end{align*}
thus the integral in \eqref{3.17} with respect to $\xi_{i}$ can be simplified to
\begin{align*}
&\int_{\xi_{i}}\frac{e^{ix_{i}\xi_{i}}1_{|\lambda-\xi_{i}|\gg\varepsilon2^{k}}}{(\xi_{i}-\lambda)^{k+1}}
\mathscr{F}_{y_{i}}\{K(y_{i},\bar{\xi}_{i},\tau)\}d\xi_{i}\\
&=\int_{y_{i}}K(y_{i},\bar{\xi}_{i},\tau)\int_{\xi_{i}}\frac{e^{i(x_{i}-y_{i})\xi_{i}}
1_{|\lambda-\xi_{i}|\gg\varepsilon2^{k}}}{(\xi_{i}-\lambda)^{k+1}}d\xi_{i}dy_{i}\\
&\lesssim \int_{y_{i}}\frac{1}{k}\varepsilon^{-k}2^{-k^{2}}e^{ix_{i}\lambda}K(y_{i},\bar{\xi}_{i},\tau)dy_{i},
\end{align*}
which yields that
\begin{align*}
u_{3}^{2}&\lesssim\sum_{k=1}^{\infty}(-1)^{k}i^{k}\frac{1}{k}\int_{\bar{\xi_{i}}\times\tau}e^{it\tau}e^{i\bar{x_{i}}\bar{\xi_{i}}}
(2\lambda)^{-k-1}1_{-\tau-|\bar{\xi}|^{2}\sim 2^{2k}} \int_{y_{i}}e^{ix_{i}\lambda}K(y_{i},\bar{\xi_{i}},\tau)dy_{i}d\bar{\xi_{i}}d\tau \\
&=:{\Gamma}_{3}^{2}.
\end{align*}
By the Plancherel equality, Bochner's equality and Lemma \ref{semi-stri}, we obtain
\begin{align}
\begin{split}
\|P_{k}u_{3}^{2}\|_{L^{2,\infty}_{\mathbf{e}_{i}}}
&=\left\|\mathscr{F}^{-1}_{\xi}\bigg\{\chi_{k}(\xi)e^{it|\xi|^{2}}\mathscr{F}_{x}\big\{u_{3}^{2}\cdot e^{-it|\xi|^{2}}\big\}\bigg\}\right\|_{L^{2,\infty}_{\mathbf{e}_{i}}} \\
&\lesssim \|e^{it|\xi|^{2}}\|_{L^{\infty}_{t,\xi}}\|e^{it\Delta}P_{k}\Gamma_{3}^{2}\|_{L^{2,\infty}_{\mathbf{e}_{i}}}
\lesssim 2^{(n-1)k/2}\|\Gamma_{3}^{2}\|_{L^{2}_{t,x}} \\
&\lesssim 2^{(n-1)k/2}2^{-2k^{2}-2k}\left\|1_{-\tau-|\bar{\xi}|^{2}\sim 2^{2k}}\int_{y_{i}}e^{ix_{1}\lambda}
K(y_{i},\bar{\xi_{i}},\tau)dy_{i}\right\|_{L^{2}_{x_{i}}L^{2}_{\bar{\xi_{i}},\tau}} \\
&\lesssim 2^{(n-1)k/2}2^{-2k^{2}-\frac{3}{2}k}|\xi|^{2k}\|J(y_{i},\bar{\xi_{i}},\tau)\|_{L^{1}_{y_{i}}L^{2}_{\bar{\xi_{i}},\tau}} \\
&\lesssim 2^{(n-2)k/2}\|J(u(t,x))\|_{L^{1,2}_{\mathbf{e}_{i}}}.\label{3.19}
\end{split}
\end{align}
For the term $u_{3}^{3}$, we use the mean value theorem of $\varphi_{3}(\tau,\xi)$ to get
\begin{align*}
|\varphi_{3}(\tau,\xi)|
\lesssim \sum_{k=1}^{\infty} \frac{\varepsilon^{k}|\xi|^{2k}(k+1)}{2^{k+1}\lambda^{k+1}|\xi_{i}-\lambda|^{k+1}}
\left(1+\frac{\lambda-\xi_{i}}{\lambda+\xi_{i}}\theta\right)^{k}\left|\frac{\lambda-\xi_{i}}{\lambda+\xi_{i}}\right|
\lesssim \sum_{k=1}^{\infty}\frac{k+1}{2^{2k}},\,\,\theta\in(0,1).
\end{align*}
The convergence is obvious, then it can be proved by using the similar method as the term $u_{1}$.
\end{proof}

\begin{lemma}[$X^{0,b,q}$ space]\label{X-linear}
Assume that $n\geq3$, $u(t,x)$ is the solution of \eqref{deri-Ginz}, $J(u(t,x))$ is the nonlinear term, for any $\varepsilon>0$,
\begin{align*}
u_{t}-(\varepsilon+i)\Delta u=J(u(t,x)),\quad u(0,x)=u_{0}.
\end{align*}
Then we have
\begin{align}
&\|P_{k}u\|_{X^{0,\frac{1}{2},\infty}_{+}}
\leq C\big(\|u_{0}\|_{L^{2}_{x}}+\|P_{k}J(u(t,x))\|_{X^{0,-\frac{1}{2},\infty}_{+}}\big),\label{3.14}\\
&\|P_{k}u\|_{Z_{k}}\leq C\big(  \varepsilon^{\frac{1}{2}}\|u_{0}\|_{L^{2}_{x}}+2^{-k}\|J(u(t,x))\|_{L^{2}_{t,x}}\big),\label{3.15}
\end{align}
where the constant $C>0$ is independent of $\varepsilon$ and $k$.
\end{lemma}

\begin{proof}
First, we show the second inequality \eqref{3.15}. By the definition of the norm $Z_{k}$ and the fact that $|\xi|\sim 2^{k}$, one has
\begin{align*}
\|P_{k}u\|_{Z_{k}}
&=2^{-k}\left\|-\mathscr{F}^{-1}_{\xi}\{\varepsilon|\xi|^{2}\chi_{k}(\xi)\hat{u}(\xi)\}
+\mathscr{F}^{-1}_{\xi}\{\chi_{k}(\xi)\hat{J}(u(s,\xi))e^{-(i+\varepsilon)|\xi|^{2}(t-s)}\}\right\|_{L^{2}_{t,x}}\\
&\lesssim \varepsilon 2^{k}\|P_{k}u\|_{L^{2}_{t,x}}+2^{-k}\|P_{k}J(u(t,x))\|_{L^{2}_{t,x}}.
\end{align*}
We estimate the first term by the plancherel equality and H\"{o}lder's inequality:
\begin{align*}
\varepsilon 2^{k}\|P_{k}u\|_{L^{2}_{t,x}}
&\lesssim \varepsilon 2^{k}\|P_{k}e^{(i+\varepsilon)t\Delta}u_{0}\|_{L^{2}_{t,x}}
+\varepsilon 2^{k} \left\|P_{k}\int_{0}^{t}e^{(i+\varepsilon)(t-s)\Delta}J(u(s))ds\right\|_{L^{2}_{t,x}}\\
&\lesssim \varepsilon 2^{k}\left[\int_{0}^{\infty}\|e^{-(i+\varepsilon)t|\xi|^{2}}\|^{2}_{L^{\infty}_{\xi}}dt\right]^{\frac{1}{2}}
\|P_{k}u_{0}\|_{L^{\infty}_{x}}\\
&\quad+\varepsilon 2^{k}\left\|\mathscr{F}^{-1}_{\xi}\big\{\chi_{k}(\xi)\sup_{s\in[0,t]}|\hat{J}(s,\xi)|\big\}
\int_{0}^{t}e^{-(i+\varepsilon)(t-s)|\xi|^{2}}ds\right\|_{L^{2}_{t,x}}\\
&\lesssim \varepsilon 2^{k}\varepsilon^{-\frac{1}{2}}|\xi|^{-1}\|P_{k}u_{0}\|_{L^{2}_{x}}
+\varepsilon 2^{k}\varepsilon^{-1}|\xi|^{-2}\|P_{k}J(u(t,x))\|_{L^{2}_{t,x}}\\
&\lesssim \varepsilon^{\frac{1}{2}}\|u_{0}\|_{L^{2}_{x}}+2^{-k}\|J(u(t,x))\|_{L^{2}_{t,x}},
\end{align*}
thus we get \eqref{3.15}.

Next, we want to show \eqref{3.14}. Assume that $J(u(t,x))=0$, since $X^{0,b,q}$ space contains the Fourier transform for time, we extend $\tilde{u}=e^{it\Delta+\varepsilon|t|\Delta}u_{0}$. By the Plancherel equality, we obtain
\begin{align*}
\|P_{k}u\|_{X^{0,\frac{1}{2},\infty}_{+}}
&=\sup_{j\in \mathbb{Z}}2^{\frac{j}{2}}\left\|\chi_{j}(\tau+|\xi|^{2})\chi_{k}(\xi)\hat{u}_{0}(\xi)
\mathscr{F}_{t}\big\{e^{-it|\xi|^{2}-\varepsilon|t||\xi|^{2}}\big\}\right\|_{L^{2}_{\tau,\xi}}\\
&\lesssim \sup_{j\in \mathbb{Z}}2^{\frac{j}{2}}\left\|\chi_{k}(\xi)\hat{u}_{0}(\xi)\left\|\chi_{j}(\tau+|\xi|^{2})
\frac{\varepsilon|\xi|^{2}}{(\varepsilon|\xi|^{2})^{2}+(\tau+|\xi|^{2})^{2}}\right\|_{L^{2}_{\tau}}\right\|_{L^{2}_{\xi}}\\
&\lesssim \|\chi_{k}(\xi)\hat{u}_{0}(\xi)\|_{L^{2}_{\xi}}
=\|u_{0}\|_{L^{2}_{x}}.
\end{align*}

We assume that $u_{0}=0$ and we define $\tilde{u}=\mathscr{F}^{-1}_{\tau,\xi}\frac{1}{\tau+|\xi|^{2}+i\varepsilon|\xi|^{2}}\mathscr{F}_{t,x}\tilde{J}(u(t,x))$,
where the term $\tilde{J}(u(t,x))$ is the extension of the term $J(u(t,x))$. Let $\tilde{J}(u(t,x))=J(u(t,x))1_{t\geq0}+J(u(t,x))1_{t<0}$ such that
\begin{align*}
\|P_{k}u\|_{X^{0,\frac{1}{2},\infty}_{+}}
&\lesssim \|P_{k}\tilde{u}\|_{X^{0,-\frac{1}{2},\infty}}.
\end{align*}
Similar to the argument of \eqref{3.2}, we get \eqref{3.14}. The proof of Lemma \ref{X-linear} is completed.
\end{proof}

\begin{proposition}[Linear estimate]\label{pro linear-es}
Assume that $n\geq3$, $u(t,x)$ is the solution of \eqref{deri-Ginz}, $J(u(t,x))$ is the nonlinear term, for any $\varepsilon>0$,
\begin{align*}
u_{t}-(\varepsilon+ai)\Delta u=J(u(t,x)),\quad u(x,0)=u_{0}.
\end{align*}
then we have
\begin{align}
\|u(t,x)\|_{F^{\frac{n}{2}}\cap Z^{\frac{n}{2}}}\leq C\big(\|u_{0}\|_{\dot{B}^{\frac{n}{2}}_{2,1}}+\|J(x,t)\|_{N^{\frac{n}{2}}}\big),\label{3.1}
\end{align}
where the constant $C>0$ is independent of $\varepsilon$ and $k$.
\end{proposition}
\begin{proof}
Combining Lemmas \ref{Leb-linear}--\ref{X-linear}, we can get the desired result of Proposition \ref{pro linear-es}.
\end{proof}
\section{The nonlinear estimate }\label{Sec4}
In this section, we establish some nonlinear estimates, where the nonlinear parts of the complex derivative Ginzburg-Landau equation are given by
\begin{align*}
J(u(t,x))&=-(1+i)(v\cdot\nabla)u-\frac{2(\varepsilon+i)\bar{u}(\nabla u)^{2}}{1+|u|^{2}}
-\frac{{\rm{Im}}F-i{\rm{Re}}H}{1+|u|^{2}}\\
&=:J_{1}(u(t,x))+J_{2}(u(t,x))+J_{3}(u(t,x)).
\end{align*}
and
\begin{align*}
F(u,\bar{u}):=\frac{2(1+\bar{u}^{2})(1-|u|^2)(v\cdot\nabla)u}{1+|u|^{2}},\quad
H(u,\bar{u}):=\frac{4(v\cdot\nabla)u(\bar{u}^{2}+|u|^{2})}{1+|u|^{2}}.
\end{align*}

\begin{lemma}[\cite{Guo}, Lemma 5.4]\label{Qjk-bdd}
$(1)$\;Assume that $\Omega$ is a Banach space with translation invariant for time-space. If $j,k\in\mathbb{Z}$ and $j\geq2k-100$, then $Q_{\leq j}P_{k}$ is bounded on $\Omega$ with bound independent of $j,k$.

$(2)$\;For any $j,k\in\mathbb{Z}$ and $1\leq p\leq\infty$, then $Q_{\leq j}P_{k,\mathbf{e}}$ is bounded on $L^{p,2}_{\mathbf{e}}$ with bound independent of $j,k$.

$(3)$\;For any $j\in\mathbb{Z}$ and $1\leq p\leq\infty$, then $Q_{\leq j}$ is bounded on $L^{p}_{t}L^{2}_{x}$ with bound independent of $j$.
\end{lemma}

We need to estimate the nonlinear term $J(u(t,x))$ in the space $N^{n/2}$, by dividing the norm of $N^{n/2}$ space into two parts, i.e.,
the norm of $L^{2}_{t,x}$ and the norm of piecewise function denoted by $N^{*}_{k}$.
Notice that the main structure of the nonlinear parts are $u^{2}(\nabla u)$ and $u(\nabla u)^{2}$, then we establish $L^{2}_{t,x}$- and $N^{*}_{k}$-estimates of both terms, respectively.
\begin{lemma}\label{Lemma4.3}
Assume that $n\geq3$, for any $k_{j}\in\mathbb{Z},\;j=1,2,3$, we have
\begin{align}
\sum_{k_{j}}2^{k_{3}(n-2)/2}\left\|P_{k_{3}}\left[\sum_{i=1}^{n}v_{i}\partial_{x_{i}}w(P_{k_{1}}fP_{k_{2}}g)\right]\right\|
_{L^{2}_{t,x}}\leq C\|v\|_{L^{\infty}_{t,x}}\|w\|_{Y^{n/2}} \|f\|_{F^{n/2}} \|g\|_{F^{n/2}}.\label{4.3}
\end{align}
where the constant $C>0$ is independent of $\varepsilon$ and $k_{j}$.
\end{lemma}

\begin{proof}
By the symmetry of $k_{1},k_{2}$, we assume that $k_{1}\leq k_{2}$. Then we get
\begin{align*}
&\sum_{k_{j}}2^{k_{3}(n-2)/2}\left\|P_{k_{3}}\left[\sum_{i=1}^{n}v_{i}\partial_{x_{i}}w(P_{k_{1}}fP_{k_{2}}g)\right]\right\|_{L^{2}_{t,x}}\\
&\leq\sum_{k_{j}}2^{k_{3}(n-2)/2}\left\|P_{k_{3}}\left[\sum_{i=1}^{n}v_{i}\partial_{x_{i}}P_{\leq k_{1}+k_{2}-10}w
(P_{k_{1}}fP_{k_{2}}g)\right]\right\|_{L^{2}_{t,x}}\\
&\quad+\sum_{k_{j}}2^{k_{3}(n-2)/2}\left\|P_{k_{3}}\left[\sum_{i=1}^{n}v_{i}\partial_{x_{i}}P_{\geq k_{1}+k_{2}-9}
w(P_{k_{1}}fP_{k_{2}}g)\right]\right\|_{L^{2}_{t,x}}\\
&=:I_1+I_2
\end{align*}

Firstly, we estimate the term $I_1$. Assume that $k_{3}\geq k_{1}+5$, we use Bernstein's inequality and H\"{o}lder's inequality to get
\begin{align}
\begin{split}
I_1
&\lesssim \sum_{k_{j}}2^{k_{3}(n-2)/2}2^{k_{3}n/2}\left\|P_{k_{3}}\left[\sum_{i=1}^{n}v_{i}\partial_{x_{i}}P_{\leq k_{1}+k_{2}-10}w
(P_{k_{1}}fP_{k_{2}}g)\right]\right\|_{L^{2}_{t}L^{1}_{x}}  \\
&\lesssim \sum_{k_{j}}2^{k_{3}(n-2)/2}2^{k_{3}n/2}\left\|P_{k_{3}}\left[\sum_{i=1}^{n}v_{i}\partial_{x_{i}}
P_{\leq k_{1}+k_{2}-10}w\right]\right\|_{L^{\infty}_{t}L^{2}_{x}}
\left\|\widetilde{P_{k_{3}}}(P_{k_{1}}fP_{k_{2}}g)\right\|_{L^{2}_{t,x}}.\label{j>k1+k2}
\end{split}
\end{align}
Notice that the operator $P_{k_{3}}$ always exerts influence on the term $P_{\leq k_{1}+k_{2}-10}w$ and the term $P_{k_{1}}fP_{k_{2}}g$ by the frequency, and the operator $\widetilde{P_{k_{3}}}$ has more larger frequency than the operator$P_{k_{3}}$. So the norm that we obtain will also be larger when we apply a larger frequency to $P_{k_{1}}fP_{k_{2}}g$ (The more accurate reason is the compact supported set becomes bigger, whereas we describe it by the frequency).

We use convolution properties and Plancherel's equality to obtain
\begin{align*}
&\left\|P_{k_{3}}\left[\sum_{i=1}^{n}v_{i}\partial_{x_{i}}
P_{\leq k_{1}+k_{2}-10}w\right]\right\|_{L^{\infty}_{t}L^{2}_{x}}\\
&\lesssim2^{k_{1}+k_{2}}\|v\|_{L^{\infty}_{t,x}}\left\|\chi_{k_{3}}(\xi)\chi_{\leq k_{1}+k_{2}-10}(\xi)
\hat{w}(\xi)\right\|_{L^{\infty}_{t}L^{2}_{\xi}} \\
&\lesssim2^{k_{1}+k_{2}}\|v\|_{L^{\infty}_{t,x}}\|P_{k_{3}}w\|_{L^{\infty}_{t}L^{2}_{x}}.
\end{align*}
Thus, we have
\begin{align*}
I_1
&\lesssim\sum_{k_{j}}2^{k_{3}n-k_{3}}2^{k_{1}+k_{2}}\|v\|_{L^{\infty}_{t,x}}\|P_{k_{3}}w\|_{L^{\infty}_{t}L^{2}_{x}}
2^{k_{1}(n-2)/2}\|P_{k_{1}}f\|_{L^{2}_{t}L^{\frac{2n}{2n-2}}_{x}}\|P_{k_{2}}g\|_{L^{\infty}_{t}L^{2}_{x}}\\
&\lesssim\sum_{k_{j}}2^{(k_{3}-k_{2})(n-1)/2}\|w\|_{Y^{n/2}} \|f\|_{F^{n/2}} \|g\|_{F^{n/2}}.
\end{align*}
Since $k_{3}\geq k_{1}+5$ and $k_{1}\leq k_{2}$, by the continuity of the frequency (the frequency is unbroken), we can know that $|k_{3}-k_{2}|\leq 5$. On the other hand, when $k_{3}\leq k_{1}+4$, we have $k_{3}-k_{2}\leq 4$, and then this case can be estimated in a similar way. Therefore, we obtain that $I_1\leq C\|v\|_{L^{\infty}_{t,x}}\|w\|_{Y^{n/2}} \|f\|_{F^{n/2}} \|g\|_{F^{n/2}}$.

Secondly, we estimate the term $I_2$. Assuming that $k_{3}\geq k_{1}+5$, using H\"{o}lder's inequality, Bernstein's inequality and convolution properties, it holds that
\begin{align*}
I_2&\lesssim\sum_{k_{j}}2^{k_{3}(n-2)/2}\left\|\sum_{i=1}^{n}v_{i}\partial_{x_{i}}P_{\geq k_{1}+k_{2}-9}
w\right\|_{L^{\infty}_{t,x}}\left\|P_{k_{1}}fP_{k_{2}}g\right\|_{L^{2}_{t,x}} \\
&\lesssim\sum_{k_{j}}2^{k_{3}(n-2)/2}\sum_{i=1}^{n}\|v\|_{L^{\infty}_{t,x}}\|P_{k_{1}+k_{2}}\partial{x_{i}w}\|_{L^{\infty}_{t,x}}
\|P_{k_{1}}f\|_{L^{2}_{t}L^{\infty}_{x}}\|P_{k_{2}}g\|_{L^{\infty}_{t}L^{2}_{x}} \\
&\lesssim\sum_{k_{j}}2^{k_{3}(n-2)/2}2^{k_{1}+k_{2}}\|v\|_{L^{\infty}_{t,x}}2^{(k_{1}+k_{2})n/2}\|P_{k_{1}+k_{2}}w\|_{L^{\infty}_{t}L^{2}_{x}}\\
&\quad\cdot2^{k_{1}(n-2)/2}\|P_{k_{1}}f\|_{L^{2}_{t}L^{\frac{2n}{2n-2}}_{x}}\|P_{k_{2}}g\|_{L^{\infty}_{t}L^{2}_{x}} \\
&\lesssim\sum_{k_{j}}2^{(k_{3}-k_{2})(n/2-1)}\|v\|_{L^{\infty}_{t,x}}\|w\|_{Y^{n/2}} \|f\|_{F^{n/2}} \|g\|_{F^{n/2}},
\end{align*}
where we have used the following fact in the second line:
\begin{align*}
&\sum_{i=1}^{n}\|v_{i}\partial_{x_{i}}P_{\geq k_{1}+k_{2}-9}w\|_{L^{\infty}_{t,x}}\\
&\lesssim\|v\|_{L^{\infty}_{t,x}}\sum_{i=1}^{n}\sum_{j:=k_{1}+k_{2}\in\mathbb{Z}}\sup_{j\geq k_{1}+k_{2}-9}
\|\partial_{x_{i}}P_{j}w\|_{L^{\infty}_{t,x}}\\
&\lesssim\|v\|_{L^{\infty}_{t,x}}\sum_{i=1}^{n}\sum_{k_{1}+k_{2}\in\mathbb{Z}}
\|\partial_{x_{i}}P_{k_{1}+k_{2}}w\|_{L^{\infty}_{t,x}}.
\end{align*}
When $k_{3}\leq k_{1}+4$, we have $k_{3}-k_{2}\leq 4$. This case can be estimated in a similar way. Hence, $I_2\leq C\|v\|_{L^{\infty}_{t,x}}\|w\|_{Y^{n/2}} \|f\|_{F^{n/2}} \|g\|_{F^{n/2}}$. The proof of \eqref{4.3} is finished.
\end{proof}

\begin{lemma}\label{Lemma4.4}
Assume that $n\geq3$, for any $k_{j}\in\mathbb{Z},\;j=1,2,3$, we have
\begin{align}
\sum_{k_{j}}2^{k_{3}n/2}\left\|P_{k_{3}}\left[\sum_{i=1}^{n}v_{i}\partial_{x_{i}}w(P_{k_{1}}fP_{k_{2}}g)\right]\right\|
_{N^{*}_{k_{3}}}\leq C\|v\|_{L^{\infty}_{t,x}}\|w\|_{Y^{n/2}} \|f\|_{F^{n/2}\cap Z^{n/2}} \|g\|_{F^{n/2}\cap Z^{n/2}},\label{4.4}
\end{align}
where the constant $C>0$ is independent of $\varepsilon$ and $k_{j}$.
\end{lemma}

\begin{proof}
By the symmetry of $k_{1}, k_{2}$, we assume that $k_{1}\leq k_{2}$. By decomposing the operator $P_{k_{3}}$, we get
\begin{align*}
&\sum_{k_{j}}2^{k_{3}n/2}\left\|P_{k_{3}}\left[\sum_{i=1}^{n}v_{i}\partial_{x_{i}}w(P_{k_{1}}fP_{k_{2}}g)\right]\right\|_{N^{*}_{k_{3}}}\\
&\leq\sum_{k_{j}}2^{k_{3}n/2}\left\|P_{k_{3}}\left[\sum_{i=1}^{n}v_{i}\partial_{x_{i}}P_{\leq k_{1}+k_{2}}w
(P_{k_{1}}fP_{k_{2}}g)\right]\right\|_{N^{*}_{k_{3}}}\\
&\quad+\sum_{k_{j}}2^{k_{3}n/2}\left\|P_{k_{3}}\left[\sum_{i=1}^{n}v_{i}\partial_{x_{i}}P_{\geq k_{1}+k_{2}}
w(P_{k_{1}}fP_{k_{2}}g)\right]\right\|_{N^{*}_{k_{3}}}\\
&=:J_1+J_2.
\end{align*}
We first estimate the term $J_2$ by decomposing the frequency. Assume that $k_{3}\geq k_{1}+20$. In order to estimate more accurately, we further assume that $-k_{2}\leq 10$, the norm on $N^{*}_{k_{3}}$ is $L_{\mathbf{e}_{j}}^{1,2}$ at this time. Then, we use H\"{o}lder's inequality and Bernstein's inequality to get{\small
\begin{align*}
&\sum_{k_{j}}2^{k_{3}n/2-k_{3}/2}\left\|P_{k_{3}}\left[\sum_{i=1}^{n}v_{i}\partial_{x_{i}}P_{\geq k_{1}+k_{2}}w
(P_{k_{1}}fP_{k_{2}}g)\right]\right\|_{L_{\mathbf{e}_{j}}^{1,2}} \\
&\lesssim \sum_{k_{j}}2^{k_{3}(n-1)/2+k_{1}+k_{2}}\|v\|_{L^{\infty}_{t,x}}\|P_{k_{1}+k_{2}}wP_{k_{2}}g\|_{L^{2}_{t,x}}
\|P_{k_{1}}f\|_{L_{\mathbf{e}_{j}}^{2,\infty}} \\
&\lesssim\sum_{k_{j}}2^{k_{3}(n-1)/2+k_{1}+k_{2}}\|v\|_{L^{\infty}_{t,x}}2^{(k_{1}+k_{2})(n-1)/2-k_{2}/2+k_{1}(n-1)/2}
\|P_{k_{1}+k_{2}}w\|_{F_{k_{1}+k_{2}}}\|P_{k_{2}}g\|_{F_{k_{2}}}\|P_{k_{1}}f\|_{F_{k_{1}}}\\
&\lesssim \sum_{k_{j}}2^{(k_{3}-k_{2})(n-1)/2-k_{2}/2}
\|v\|_{L^{\infty}_{t,x}}\|P_{k_{1}+k_{2}}w\|_{Y^{n/2}}\|P_{k_{1}}f\|_{F^{n/2}\cap Z^{n/2}}\|P_{k_{2}}g\|_{F^{n/2}\cap Z^{n/2}}.
\end{align*}}
On the other hand, assuming that $-k_{2}\geq 9$ and decomposing the modulation operator, we have
\begin{align*}
J_2&\leq \sum_{k_{j}}2^{k_{3}n/2}\left\|P_{k_{3}}\left[\sum_{i=1}^{n}v_{i}\partial_{x_{i}}P_{\geq k_{1}+k_{2}-9}Q_{\geq k_{1}+k_{2}-10}
w(P_{k_{1}}fP_{k_{2}}g)\right]\right\|_{N^{*}_{k_{3}}}\\
&\quad+\sum_{k_{j}}2^{k_{3}n/2}\left\|P_{k_{3}}\left[\sum_{i=1}^{n}v_{i}\partial_{x_{i}}P_{\geq k_{1}+k_{2}-9}Q_{\leq k_{1}+k_{2}-11}
w(P_{k_{1}}fP_{k_{2}}g)\right]\right\|_{N^{*}_{k_{3}}}\\
&=:J_{21}+J_{22}.
\end{align*}
It is easy to verify that the product of Littlewood-Paley projectors $P_{k}$ and modulation projectors $Q_{k}$ are symmetrical, i.e.,
\begin{align*}
P_{k}Q_{k}u(x,t)=Q_{k}P_{k}u(x,t).
\end{align*}

Now, we estimate the term $J_{21}$. Noticing that the norm is $L_{t}^{1}L_{x}^{2}$ on $N^{*}_{k_{3}}$, employing H\"{o}lder's inequality, Minkowski's inequality and Bernstein's inequality, one has {\small
\begin{align*}
J_{21}
&=\sum_{k_{j}}2^{k_{3}n/2}\left\|P_{k_{3}}\left[\sum_{i=1}^{n}v_{i}\partial_{x_{i}}P_{\geq k_{1}+k_{2}-9}Q_{\geq k_{1}+k_{2}-10}
w(P_{k_{1}}fP_{k_{2}}g)\right]\right\|_{L_{t}^{1}L_{x}^{2}}\\
&\lesssim\sum_{k_{j}}2^{k_{3}n/2}\|v\|_{L^{\infty}_{t,x}}\sum_{i=1}^{n}\sum_{k_{1}+k_{2}\in\mathbb{Z}}\|P_{k_{1}+k_{2}}
Q_{\geq k_{1}+k_{2}-10}\partial_{x_{i}}w\|_{_{L_{t}^{2}L_{x}^{\infty}}}
\|P_{k_{1}}f\|_{L_{\mathbf{e}_{j}}^{2,\infty}}\|P_{k_{2}}g\|_{L_{\mathbf{e}_{j}}^{\infty,2}} \\
&\lesssim \sum_{k_{j}}2^{k_{3}n/2}\|v\|_{L^{\infty}_{t,x}}2^{(k_{1}+k_{2})n/2}2^{k_{1}+k_{2}}\sum_{j\geq k_{1}+k_{2}-10}\|Q_{j}(P_{k_{1}+k_{2}}w)\|_{_{L_{t,x}^{2}}}\|P_{k_{1}}f\|_{L_{\mathbf{e}_{j}}^{2,\infty}}
\|P_{k_{2}}g\|_{L_{\mathbf{e}_{j}}^{\infty,2}} \\
&\lesssim \sum_{k_{j}}2^{k_{3}n/2+(k_{1}+k_{2})n/2+k_{1}+k_{2}}\|v\|_{L^{\infty}_{t,x}}
\left[\sum_{j\geq k_{1}+k_{2}-10}\left(2^{j}\|Q_{j}(P_{k_{1}+k_{2}}w)\|_{_{L_{t,x}^{2}}}\right)^{2}\right]^{\frac{1}{2}}\\
&\quad\left[\sum_{j\geq k_{1}+k_{2}-10}2^{-2j}\right]^{\frac{1}{2}}
\|P_{k_{1}}f\|_{L_{\mathbf{e}_{j}}^{2,\infty}}\|P_{k_{2}}g\|_{L_{\mathbf{e}_{j}}^{\infty,2}}\\
&\lesssim\sum_{k_{j}}2^{k_{3}n/2+(k_{1}+k_{2})n/2+k_{1}+k_{2}}\|v\|_{L^{\infty}_{t,x}}2^{-(k_{1}+k_{2})}
\|P_{k_{1}+k_{2}}w\|_{X^{0,1,2}}\|P_{k_{1}}f\|_{L_{\mathbf{e}_{j}}^{2,\infty}}\|P_{k_{2}}g\|_{L_{\mathbf{e}_{j}}^{\infty,2}} \\
&\lesssim \sum_{k_{j}}2^{(k_{3}-k_{2})n/2+(k_{1}+k_{2})/2}
\|v\|_{L^{\infty}_{t,x}}\|P_{k_{1}+k_{2}}w\|_{Y^{n/2}}\|P_{k_{1}}f\|_{F^{n/2}}\|P_{k_{2}}g\|_{F^{n/2}}.
\end{align*}}
\!\!For the term $J_{22}$, divide $P_{k_{1}}fP_{k_{2}}g$ into two parts by modulation projectors and then we obtain {\small
\begin{align*}
J_{22}
&\leq \sum_{k_{j}}2^{k_{3}n/2}\sum_{k_{j}}2^{k_{3}n/2}\left\|P_{k_{3}}\left[\sum_{i=1}^{n}v_{i}\partial_{x_{i}}P_{\geq k_{1}+k_{2}-9}
Q_{\leq k_{1}+k_{2}-10}wQ_{\geq k_{1}+k_{2}+40}(P_{k_{1}}fP_{k_{2}}g)\right]\right\|_{N^{*}_{k_{3}}}\\
&\quad+\sum_{k_{j}}2^{k_{3}n/2}\sum_{k_{j}}2^{k_{3}n/2}\left\|P_{k_{3}}\left[\sum_{i=1}^{n}v_{i}\partial_{x_{i}}
P_{\geq k_{1}+k_{2}-9}Q_{\leq k_{1}+k_{2}-11}wQ_{\leq k_{1}+k_{2}-39}(P_{k_{1}}fP_{k_{2}}g)\right]\right\|_{N^{*}_{k_{3}}}\\
&=:J_{221}+J_{222}.
\end{align*}}

Next, we estimate the term $J_{221}$. Similarly, one has {\small
\begin{align*}
J_{221}
&=\sum_{k_{j}}2^{k_{3}n/2}\left\|P_{k_{3}}\left[\sum_{i=1}^{n}v_{i}\partial_{x_{i}}P_{\geq k_{1}+k_{2}-9}
Q_{\leq k_{1}+k_{2}-10}wQ_{\geq k_{1}+k_{2}+40}(P_{k_{1}}fP_{k_{2}}g)\right]\right\|_{X_{+}^{0,-\frac{1}{2},1}} \\
&\lesssim \sum_{k_{j}}2^{k_{3}n/2}\sum_{j\geq k_{1}+k_{2}}2^{-\frac{j}{2}}\left\|P_{k_{3}}\left(\sum_{i=1}^{n}v_{i}
\partial_{x_{i}}P_{\geq k_{1}+k_{2}-9}Q_{\leq k_{1}+k_{2}-10}wQ_{\geq k_{1}+k_{2}+40}
(P_{k_{1}}fP_{k_{2}}g)\right)\right\|_{L_{t,x}^{2}} \\
&\lesssim \sum_{k_{j}}2^{k_{3}n/2}2^{k_{1}+k_{2}}\|v\|_{L^{\infty}_{t,x}}2^{(k_{1}+k_{2})n/2}\|P_{k_{1}+k_{2}}Q_{\leq k_{1}+k_{2}-10}w\|
_{L^{\infty}_{t}L^{2}_{x}}\|P_{k_{1}}fP_{k_{2}}g\|_{L^{2}_{t,x}} \\
&\lesssim \sum_{k_{j}}2^{k_{3}n/2}2^{(k_{1}+k_{2})(n+1)/2}\|v\|_{L^{\infty}_{t,x}}\|P_{k_{1}+k_{2}}w\|_{Y_{k_{1}+k_{2}}}
2^{(n-1)k_{1}/2}\|P_{k_{1}}f\|_{F_{k_{1}}}2^{-k_{2}/2}\|P_{k_{2}}f\|_{F_{k_{2}}}  \\
&\lesssim \sum_{k_{j}}2^{(k_{3}-k_{2})n/2}\|v\|_{L^{\infty}_{t,x}}
\|P_{k_{1}+k_{2}}w\|_{Y^{n/2}}\|P_{k_{1}}f\|_{F^{n/2}}\|P_{k_{2}}g\|_{F^{n/2}},
\end{align*}}
\!\!where the boundness of the operators $Q_{k},\;Q_{\geq k}$ in space $L^{2}_{t,x}$ can be obtained by Plancherel theorem and the operator $Q_{\leq k}$ on space $L^{p}_{t}L^{2}_{x}$ is bounded  by Lemma \ref{Qjk-bdd}. According to the frequency (support) of operators $Q_{\leq k_{1}+k_{2}-10}$ and $Q_{\geq k_{1}+k_{2}+40}$, here $j\geq k_{1}+k_{2}$ is obvious.

For the term $J_{222}$, by the fact that the norm on $N^{*}_{k_{3}}$ is $L^{1,2}_{\mathbf{e}_{j}}$, we have{\small
\begin{align}
\begin{split}
&J_{222}=\sum_{k_{j}}2^{k_{3}(n-1)/2}\left\|P_{k_{3}}\left[\sum_{i=1}^{n}v_{i}\partial_{x_{i}}P_{\geq k_{1}+k_{2}-9}
Q_{\leq k_{1}+k_{2}-10}wQ_{\leq k_{1}+k_{2}-39}(P_{k_{1}}fP_{k_{2}}g)\right]\right\|_{L^{1,2}_{\mathbf{e}_{j}}} \\
&\lesssim \sum_{k_{j}}2^{k_{3}(n-1)/2}\left\|P_{k_{3}}\left[\sum_{i=1}^{n}v_{i}\partial_{x_{i}}P_{\geq k_{1}+k_{2}-9}
Q_{\leq k_{1}+k_{2}-10}w\right]\right\|_{L^{2,\infty}_{\mathbf{e}_{j}}}\left\|\widetilde{P_{k_{3}}}[Q_{\leq k_{1}+k_{2}-39}(P_{k_{1}}fP_{k_{2}}g)]\right\|_{L^{2}_{t,x}} \\
&\lesssim \sum_{k_{j}}2^{k_{3}(n-1)}\|v\|_{L^{\infty}_{t,x}}\left\|P_{k_{3}}\left[\sum_{i=1}^{n}
\partial_{x_{i}}P_{\geq k_{1}+k_{2}-9}Q_{\leq k_{1}+k_{2}-10}w\right]\right\|_{L^{\infty}_{t}L^{2}_{x}}
\|P_{k_{1}}fP_{k_{2}}g\|_{L^{2}_{t,x}} \\
&\lesssim \sum_{k_{j}}2^{k_{3}(n-1)}\|v\|_{L^{\infty}_{t,x}}\|P_{k_{3}}(P_{k_{1}+k_{2}}\partial_{x_{i}}w)\|_{L^{\infty}_{t}L^{2}_{x}}
\|P_{k_{1}}fP_{k_{2}}g\|_{L^{2}_{t,x}} \\
&\lesssim \sum_{k_{j}}2^{k_{3}(n-1)}\|v\|_{L^{\infty}_{t,x}}2^{k_{1}+k_{2}}\|P_{k_{3}}w\|_{L^{\infty}_{t}L^{2}_{x}}
2^{(n-1)k_{1}/2}\|P_{k_{1}}f\|_{{L^{\infty}_{t}L^{\frac{2n}{2n-2}}_{x}}}\|P_{k_{2}}g\|_{{L^{\infty}_{t}L^{2}_{x}}} \\
&\lesssim \sum_{k_{j}}2^{(k_{3}-k_{2})(n/2-1)}\|v\|_{L^{\infty}_{t,x}}
\|P_{k_{3}}w\|_{Y^{n/2}}\|P_{k_{1}}f\|_{F^{n/2}}\|P_{k_{2}}g\|_{F^{n/2}},\label{4.11}
\end{split}
\end{align}}
\!\!where we have used Bernstein's inequality in the third line of \eqref{4.11} with respect to $\bar{x}_{j}$, and we have used Lemma \ref{Qjk-bdd} and \eqref{j>k1+k2} in the fourth line.

On the contrary, we assume that $k_{3}\leq k_{1}+19$ and repeat the above arguments. However the assumptions are changed, supposing that $-k_{1}\leq 10$ and $-k_{1}\geq 9$, respectively. Then, we obtain the desired result $J_{2}\leq C \|v\|_{L^{\infty}_{t,x}}\|w\|_{Y^{n/2}} \|f\|_{F^{n/2}\cap Z^{n/2}} \|g\|_{F^{n/2}\cap Z^{n/2}}$.

Next, we estimate the term $J_1$. Assuming that $k_{3}\geq k_{1}+20$, using the fact that the norm is $L^{1,2}_{\mathbf{e}_{j}}$ on $N^{*}_{k_{3}}$ at this time yields that
\begin{align*}
J_1
&=\sum_{k_{j}}2^{k_{3}n/2}2^{-k_{3}/2}\left\|P_{k_{3}}\left[\sum_{i=1}^{n}v_{i}\partial_{x_{i}}P_{\leq k_{1}+k_{2}}w
(P_{k_{1}}fP_{k_{2}}g)\right]\right\|_{L^{1,2}_{\mathbf{e}_{j}}} \\
&\lesssim \sum_{k_{j}}2^{k_{3}n/2}2^{-k_{3}/2}\left\|P_{k_{3}}\left[\sum_{i=1}^{n}v_{i}\partial_{x_{i}}P_{\leq k_{1}+k_{2}}
w\right]\right\|_{L^{2,\infty}_{\mathbf{e}_{j}}}\|\widetilde{P_{k_{3}}}(P_{k_{1}}fP_{k_{2}}g)\|_{L^{2}_{t,x}}  \\
&\lesssim \sum_{k_{j}}2^{k_{3}n/2}2^{-k_{3}/2}\|v\|_{L^{\infty}_{t,x}}2^{(n-1)k_{3}/2}\left\|P_{k_{3}}\left[\sum_{i=1}^{n}v_{i}
\partial_{x_{i}}P_{k_{1}+k_{2}}w\right]\right\|_{L^{\infty}_{t}L^{2}_{x}}
\|P_{k_{1}}fP_{k_{2}}g\|_{L^{2}_{t,x}} \\
&\lesssim \sum_{k_{j}}2^{k_{3}(n-1)}\|v\|_{L^{\infty}_{t,x}}2^{k_{1}+k_{2}}\|P_{k_{3}}w\|_{L^{\infty}_{t}L^{2}_{x}}
2^{(n-1)k_{1}/2}\|P_{k_{1}}f\|_{{L^{\infty}_{t}L^{\frac{2n}{2n-2}}_{x}}}\|P_{k_{2}}g\|_{{L^{\infty}_{t}L^{2}_{x}}} \\
&\lesssim \sum_{k_{j}}2^{(k_{3}-k_{2})(n/2-1)}\|v\|_{L^{\infty}_{t,x}}
\|P_{k_{3}}w\|_{Y^{n/2}}\|P_{k_{1}}f\|_{F^{n/2}}\|P_{k_{2}}g\|_{F^{n/2}}.
\end{align*}
On the other hand, assume that $k_{3}\leq k_{1}+19$, exchange the space of terms $P_{k_{1}}f$ and $P_{k_{2}}g$. In conclusion, we obtain \eqref{4.4}.
\end{proof}

\begin{lemma}\label{Lemma4.1}
Assume that $n\geq3$, for any $k_{j}\in\mathbb{Z},\;j=1,2,3$, we have
\begin{align}
\sum_{k_{j}}2^{k_{3}(n-2)/2}\left\|P_{k_{3}}\left[w\sum_{i=1}^{n}(\partial_{x_{i}}P_{k_{1}}f \partial_{x_{i}}P_{k_{2}}g)\right]\right\|
_{L^{2}_{t,x}}\leq C\|w\|_{Y^{n/2}} \|f\|_{F^{n/2}} \|g\|_{F^{n/2}}.\label{4.1}
\end{align}
where the constant $C>0$ is independent of $\varepsilon$ and $k_{j}$.
\end{lemma}

\begin{proof}
By the symmetry of $k_{1},k_{2}$, assume that $k_{1}\leq k_{2}$. We use the operator $P_{k}$ to divide the frequency of $w$ into two parts:
\begin{align*}
&\sum_{k_{j}}2^{k_{3}(n-2)/2}\left\|P_{k_{3}}\left[w\sum_{i=1}^{n}(\partial_{x_{i}}P_{k_{1}}f \partial_{x_{i}}P_{k_{2}}g)\right]\right\|
_{L^{2}_{t,x}} \\
&\leq \sum_{k_{j}}2^{k_{3}(n-2)/2}\left\|P_{k_{3}}\left[P_{\geq k_{3}-10}w\sum_{i=1}^{n}(\partial_{x_{i}}P_{k_{1}}f \partial_{x_{i}}P_{k_{2}}g)\right]\right\|_{L^{2}_{t,x}} \\
&\quad+\sum_{k_{j}}2^{k_{3}(n-2)/2}\left\|P_{k_{3}}\left[P_{\leq k_{3}-11}w\sum_{i=1}^{n}(\partial_{x_{i}}P_{k_{1}}f \partial_{x_{i}}P_{k_{2}}g)\right]\right\|_{L^{2}_{t,x}}\\
&=:K_1+K_2,
\end{align*}

Now, we estimate the term $K_1$ and assume that $k_{3}\leq k_{2}+20$. By Bernstein's inequality and convolution properties, one has
\begin{align*}
K_1&\lesssim \sum_{k_{j}}2^{(n-1)k_{3}}\|P_{\geq k_{3}-10}w\|_{L^{\infty}_{t}L^{2}_{x}}\|\sum_{i=1}^{n}(\partial_{x_{i}}P_{k_{1}}f \partial_{x_{i}}P_{k_{2}}g)\|_{L^{2}_{t,x}}\\
&\lesssim \sum_{k_{j}}2^{(n-1)k_{3}}\|P_{k_{3}}w\|_{L^{\infty}_{t}L^{2}_{x}}2^{k_{1}+k_{2}}2^{(n-2)k_{1}/2}
\|P_{k_{1}}f\|_{L^{2}_{t}L^{\frac{2n}{n-2}}_{x}}\|P_{k_{2}}g\|_{L^{\infty}_{t}L^{2}_{x}}\\
&\lesssim \sum_{k_{j}}2^{(k_{3}-k_{2})(n/2-1)}\|w\|_{Y^{n/2}}\|f\|_{F^{n/2}}\|g\|_{F^{n/2}}.
\end{align*}
On the other hand, assuming that $k_{3}\geq k_{2}+21$, since $P_{\geq k_{3}}w$ has more bigger frequency than output frequency, and the frequency of $w$ is equivalent to $2^{k_{3}}$, we obtain
\begin{align*}
K_1&\lesssim \sum_{k_{j}}2^{k_{3}(n-2)/2}\|P_{k_{3}}w\|_{L^{\infty}_{t}L^{2}_{x}}\|\sum_{i=1}^{n}(\partial_{x_{i}}P_{k_{1}}f \partial_{x_{i}}P_{k_{2}}g)\|_{L^{2}_{t}L^{\infty}_{x}} \\
&\lesssim \sum_{k_{j}}2^{k_{3}(n-2)/2}\|P_{k_{3}}w\|_{L^{\infty}_{t}L^{2}_{x}}2^{k_{2}n/2}2^{k_{1}+k_{2}}\|P_{k_{1}}f P_{k_{2}}g\|_{L^{2}_{t,x}} \\
&\lesssim \sum_{k_{j}}2^{k_{3}(n-2)/2}\|P_{k_{3}}w\|_{L^{\infty}_{t}L^{2}_{x}}2^{k_{2}n/2}2^{k_{1}+k_{2}}
\|P_{k_{1}}f\|_{L^{2,\infty}_{\mathbf{e}_{j}}}\|P_{k_{2}}g\|_{L^{\infty,2}_{\mathbf{e}_{j}}}\\
&\lesssim \sum_{k_{j}}2^{k_{2}-k_{3}}\|w\|_{Y^{n/2}}\|f\|_{F^{n/2}}\|g\|_{F^{n/2}},
\end{align*}
where the output frequency of the third line is equivalent to $2^{k_{1}}+2^{k_{2}}$ by applying the convolution theorem, and we enlarge the output frequency to $2^{k_{2}}$. Hence, we get that $K_{1}\leq C\|w\|_{Y^{n/2}}\|f\|_{F^{n/2}}\|g\|_{F^{n/2}}$.

Next, we estimate the term $K_2$ when assume that $k_{3}\leq k_{2}+5$. Using Bernstein's inequality and convolution properties yields that
\begin{align*}
K_2&\lesssim \sum_{k_{j}}2^{k_{3}(n-1)}\|P_{k_{3}}w\|_{L^{\infty}_{t}L^{2}_{x}}\|\widetilde{P_{k_{3}}}(\sum_{i=1}^{n}\partial_{x_{i}}P_{k_{1}}f \partial_{x_{i}}P_{k_{2}}g)\|_{L^{2}_{t,x}} \\
&\lesssim \sum_{k_{j}}2^{k_{3}(n-1)}\|P_{k_{3}}w\|_{L^{\infty}_{t}L^{2}_{x}}2^{k_{1}+k_{2}}2^{(n-2)k_{2}/2}
\|P_{k_{1}}f\|_{L^{\infty}_{t}L^{2}_{x}}\|P_{k_{2}}g\|_{L^{2}_{t}L^{\frac{2n}{n-2}}_{x}}\\
&\lesssim \sum_{k_{j}}2^{(k_{3}-k_{1})(n/2-1)}\|w\|_{Y^{n/2}}\|f\|_{F^{n/2}}\|g\|_{F^{n/2}}.
\end{align*}
On the other hand, assuming that $k_{3}\geq k_{1}+6$ yields that $|k_{3}-k_{2}|\leq6$. Then, one has
\begin{align*}
K_2&\lesssim \sum_{k_{j}}2^{k_{3}(n-1)}\|P_{k_{3}}w\|_{L^{\infty}_{t}L^{2}_{x}}\|\widetilde{P_{k_{3}}}(\sum_{i=1}^{n}\partial_{x_{i}}P_{k_{1}}f \partial_{x_{i}}P_{k_{2}}g)\|_{L^{2}_{t,x}} \\
&\lesssim \sum_{k_{j}}2^{k_{3}(n-1)}\|P_{k_{3}}w\|_{L^{\infty}_{t}L^{2}_{x}}2^{k_{1}+k_{2}}2^{(n-2)k_{1}/2}
\|P_{k_{1}}f\|_{L^{2}_{t}L^{\frac{2n}{n-2}}_{x}}\|P_{k_{2}}g\|_{L^{\infty}_{t}L^{2}_{x}}\\
&\lesssim \sum_{k_{j}}2^{(k_{3}-k_{2})(n/2-1)}\|w\|_{Y^{n/2}}\|f\|_{F^{n/2}}\|g\|_{F^{n/2}},
\end{align*}
which implies that $K_2\leq C \|w\|_{Y^{n/2}}\|f\|_{F^{n/2}}\|g\|_{F^{n/2}}$. Combining the above estimates of $K_1$ and $K_2$,  we get \eqref{4.1}. The proof of Lemma \ref{Lemma4.1} is completed.
\end{proof}

\begin{lemma}\label{Lemma4.2}
Assume that $n\geq3$, for any $k_{j}\in\mathbb{Z},\;j=1,\ldots,4$, we have
\begin{align}
\sum_{k_{j}}2^{k_{4}n/2}\left\|P_{k_{4}}\left[P_{k_{1}}w\sum_{i=1}^{n}(\partial_{x_{i}}P_{k_{2}}f \partial_{x_{i}}P_{k_{3}}g)\right]\right\|
_{N_{k_{4}}^{*}}\leq C\|w\|_{Y^{n/2}} \|f\|_{F^{n/2}\cap Z^{n/2}} \|g\|_{F^{n/2}\cap Z^{n/2}},\label{4.2}
\end{align}
where the constant $C>0$ is independent of $\varepsilon$ and $k_{j}$.
\end{lemma}

\begin{proof}
By the symmetry of $k_{2}$ and $k_{3}$, we assume that $k_{2}\leq k_{3}$. Notice that when $k_{4}\leq k_{1}+40$, the norm is $L_{t}^{1}L_{x}^{2}$. Utilizing H\"{o}lder's inequality and Bernstein's inequality, one has
\begin{align*}
&\sum_{k_{j}}2^{k_{4}n/2}\left\|P_{k_{4}}\left[P_{k_{1}}w\sum_{i=1}^{n}(\partial_{x_{i}}P_{k_{2}}f \partial_{x_{i}}P_{k_{3}}g)\right]\right\|_{L^{1}_{t}L^{2}_{x}}\\
&\lesssim \sum_{k_{j}}2^{k_{4}n/2}2^{k_{2}+k_{3}}\|P_{k_{1}}wP_{k_{2}}f\|_{L^{2}_{t,x}}\|P_{k_{3}}g\|_{L^{2}_{t}L^{\infty}_{x}}\\
&\lesssim \sum_{k_{j}}2^{k_{4}n/2}2^{k_{2}+k_{3}}\|P_{k_{1}}w\|_{L^{\infty}_{t}L^{2}_{x}}2^{k_{2}(n-2)/2}\|P_{k_{2}}f\|
_{L^{2}_{t}L^{\frac{2n}{n-2}}_{x}}2^{k_{3}(n-2)/2}\|P_{k_{3}}g\|_{L^{2}_{t}L^{\frac{2n}{n-2}}_{x}}\\
&\lesssim \sum_{k_{j}}2^{(k_{4}-k_{1})n/2}\|w\|_{Y^{n/2}}\|f\|_{F^{n/2}}\|g\|_{F^{n/2}}.
\end{align*}

On the other hand, we assume that $k_{4}\geq k_{1}+41$. For a more accurate estimate, we divide it into two cases with the further assumptions of $k_{2}$.

\underline{\bf Case~1.} When $k_{2}\leq k_{1}+20$, the norm is $L^{1,2}_{\mathbf{e}_{i}}$. By employing H\"{o}lder's inequality and Bernstein's inequality again, we obtain
\begin{align*}
&\sum_{k_{j}}2^{k_{4}n/2}2^{-k_{4}/2}\left\|P_{k_{4}}\left[P_{k_{1}}w\sum_{i=1}^{n}(\partial_{x_{i}}P_{k_{2}}f \partial_{x_{i}}P_{k_{3}}g)\right]\right\|_{L^{1,2}_{\mathbf{e}_{i}}}\\
&\lesssim \sum_{k_{j}}\sum_{i=1}^{n}2^{k_{4}(n-1)/2}\|P_{k_{1}}w\partial_{x_{i}}P_{k_{3}}g\|_{L^{2}_{t,x}}
\|\partial_{x_{i}}P_{k_{2}}f\|_{L^{2,\infty}_{\mathbf{e}_{i}}}\\
&\lesssim \sum_{k_{j}}2^{k_{4}(n-1)/2}2^{k_{2}+k_{3}}\|P_{k_{1}}w\|_{L^{2,\infty}_{\mathbf{e}_{i}}}
\|P_{k_{3}}g\|_{L^{\infty,2}_{\mathbf{e}_{i}}}\|P_{k_{2}}f\|_{L^{2,\infty}_{\mathbf{e}_{i}}}\\
&\lesssim \sum_{k_{j}}2^{(k_{4}-k_{3})(n-1)/2+(k_{2}-k_{1})/2}\|w\|_{Y^{n/2}}\|f\|_{F^{n/2}}\|g\|_{F^{n/2}}.
\end{align*}

\underline{\bf Case~2.} When $k_{2}\geq k_{1}+21$, we decompose $P_{k_{1}}fP_{k_{2}}g$ by modulation projectors $Q_{k}$ to get
\begin{align*}
&\sum_{k_{j}}2^{k_{4}n/2}2^{-k_{4}/2}\left\|P_{k_{4}}\left[P_{k_{1}}w\sum_{i=1}^{n}(\partial_{x_{i}}P_{k_{2}}f \partial_{x_{i}}P_{k_{3}}g)\right]\right\|_{N^{*}_{k_{4}}}\\
&\leq \sum_{k_{j}}2^{k_{4}n/2}2^{-k_{4}/2}\left\|P_{k_{4}}\left[P_{k_{1}}w Q_{\leq k_{2}+k_{3}}\sum_{i=1}^{n}(\partial_{x_{i}}P_{k_{2}}f \partial_{x_{i}}P_{k_{3}}g)\right]\right\|_{N^{*}_{k_{4}}}\\
&\quad+\sum_{k_{j}}2^{k_{4}n/2}2^{-k_{4}/2}\left\|P_{k_{4}}\left[P_{k_{1}}wQ_{\geq k_{2}+k_{3}+1}\sum_{i=1}^{n}(\partial_{x_{i}}P_{k_{2}}f \partial_{x_{i}}P_{k_{3}}g)\right]\right\|_{N^{*}_{k_{4}}}\\
&=:H_1+H_2.
\end{align*}

Now we estimate the terms $H_1$ and $H_2$ one by one. For the term $H_2$, decomposing $P_{k_{1}}w$ by modulation projectors, we get
\begin{align*}
H_2&\leq \sum_{k_{j}}2^{k_{4}n/2}2^{-k_{4}/2}\left\|P_{k_{4}}\left[P_{k_{1}}Q_{\geq k_{2}+k_{3}-10}wQ_{\geq k_{2}+k_{3}+1}\sum_{i=1}^{n}(\partial_{x_{i}}P_{k_{2}}f \partial_{x_{i}}P_{k_{3}}g)\right]\right\|_{N^{*}_{k_{4}}}\\
&\quad+\sum_{k_{j}}2^{k_{4}n/2}2^{-k_{4}/2}\left\|P_{k_{4}}\left[P_{k_{1}}Q_{\leq k_{2}+k_{3}-9}wQ_{\geq k_{2}+k_{3}+1}\sum_{i=1}^{n}(\partial_{x_{i}}P_{k_{2}}f \partial_{x_{i}}P_{k_{3}}g)\right]\right\|_{N^{*}_{k_{4}}}\\
&=:H_{21}+H_{22}.
\end{align*}
For the term $H_{21}$, by the fact that the norm on $N^{*}_{k_{4}}$ is $L_{t}^{1}L_{x}^{2}$, H\"{o}lder's inequality, Minkowski's inequality and Bernstein's inequality, one has
\begin{align*}
H_{21}&\lesssim \sum_{k_{j}}2^{k_{4}n/2}\|P_{k_{1}}Q_{\geq k_{2}+k_{3}-10}w\|_{L^{2}_{t}L^{\infty}_{x}}
\|Q_{\geq k_{2}+k_{3}+1}\sum_{i=1}^{n}(\partial_{x_{i}}P_{k_{2}}f \partial_{x_{i}}P_{k_{3}}g)\|_{L^{2}_{t,x}}\\
&\lesssim \sum_{k_{j}}2^{k_{4}n/2}2^{k_{1}n/2}\sum_{j\geq k_{2}+k_{3}-10}\|Q_{j}(P_{k_{1}}w)\|_{L^{2}_{t,x}}
2^{k_{2}+k_{3}}\|P_{k_{2}}fP_{k_{3}}g\|_{L^{2}_{t,x}}\\
&\lesssim \sum_{k_{j}}2^{k_{4}n/2}2^{k_{1}n/2}2^{-(k_{2}+k_{3})}\|P_{k_{1}}w\|_{X^{0,1,2}}2^{k_{2}+k_{3}}
\|P_{k_{2}}f\|_{L^{2,\infty}_{\mathbf{e}_{i}}}\|P_{k_{3}}g\|_{L^{\infty,2}_{\mathbf{e}_{i}}}\\
&\lesssim \sum_{k_{j}}2^{(k_{4}-k_{3})n/2+(k_{1}-k_{2})}\|w\|_{Y^{n/2}}\|f\|_{F^{n/2}}\|g\|_{F^{n/2}}.
\end{align*}
Noted that by the hypothesis that $k_{3}\leq k_{1}+21$ and $k_{4}\leq k_{1}+41$, we get $|k_{4}-k_{3}|\leq 20$, because the continuity of the frequency (the frequency is unbroken).

For the term $H_{22}$, using the fact that the norm is $X^{0,-1/2,1}_{+}$ on $N^{*}_{k_{3}}$ and Lemma \ref{Qjk-bdd}, it holds that
\begin{align*}
H_{22}
&\lesssim \sum_{k_{j}}2^{k_{4}n/2}\sum_{j\geq k_{2}+k_{3}}2^{-\frac{j}{2}}\|P_{k_{1}}Q_{\leq k_{2}+k_{3}-10}w\|_{L^{\infty}_{t,x}}
\|Q_{\geq k_{2}+k_{3}+1}\sum_{i=1}^{n}(\partial_{x_{i}}P_{k_{2}}f \partial_{x_{i}}P_{k_{3}}g)\|_{L^{2}_{t,x}}\\
&\lesssim \sum_{k_{j}}2^{k_{4}n/2}2^{-(k_{2}+k_{3})/2}2^{k_{1}n/2}\|P_{k_{1}}w\|_{L^{\infty}_{t}L^{2}_{x}}
2^{k_{2}+k_{3}}\|P_{k_{2}}f\|_{L^{2,\infty}_{\mathbf{e}_{i}}}\|P_{k_{3}}g\|_{L^{\infty,2}_{\mathbf{e}_{i}}}\\
&\lesssim \sum_{k_{j}}2^{(k_{4}-k_{3})n/2}\|w\|_{Y^{n/2}}\|f\|_{F^{n/2}}\|g\|_{F^{n/2}},
\end{align*}
where the frequency of $Q_{\leq k_{2}+k_{3}-10}wQ_{\geq k_{2}+k_{3}+1}(P_{k_{2}}f P_{k_{3}}g)$ can be obtained by the support from convolution, i.e., $j\geq k_{2}+k_{3}$. Combining the above estimates, we can easily know that $H_2\leq C\|w\|_{Y^{n/2}}\|f\|_{F^{n/2}}\|g\|_{F^{n/2}}$.

Next, we estimate the term $H_1$. Similarly, we employ the operator $Q_{k}$ for decomposing the frequency, and then we get
\begin{align*}
H_1&\leq \sum_{k_{j}}2^{k_{4}n/2}\left\|P_{k_{4}}\left[P_{k_{1}}w Q_{\leq k_{2}+k_{3}}\sum_{i=1}^{n}(\partial_{x_{i}}P_{k_{2}}
Q_{\geq k_{2}+k_{3}+40}f \partial_{x_{i}}P_{k_{3}}g)\right]\right\|_{N^{*}_{k_{4}}}\\
&\quad+\sum_{k_{j}}2^{k_{4}n/2}\left\|P_{k_{4}}\left[P_{k_{1}}w Q_{\leq k_{2}+k_{3}}\sum_{i=1}^{n}(\partial_{x_{i}}P_{k_{2}}
Q_{\leq k_{2}+k_{3}+39}f \partial_{x_{i}}P_{k_{3}}g)\right]\right\|_{N^{*}_{k_{4}}}\\
&=:H_{11}+H_{12}.
\end{align*}

For the term $H_{11}$, by the fact that the norm on $N^{*}_{k_{3}}$ is $L^{1,2}_{\mathbf{e}_{j}}$, applying the Minkowski's inequality and Bernstein's inequality, one has
\begin{align*}
H_{11}&\lesssim \sum_{k_{j}}2^{k_{4}n/2}2^{-k_{4}/2}\left\|P_{k_{4}}\left[P_{k_{1}}w Q_{\leq k_{2}+k_{3}}\sum_{i=1}^{n}
(\partial_{x_{i}}P_{k_{2}}Q_{\geq k_{2}+k_{3}+40}f \partial_{x_{i}}P_{k_{3}}g)\right]\right\|_{L^{1,2}_{\mathbf{e}_{j}}}\\
&\lesssim \sum_{k_{j}}2^{k_{4}(n-1)/2}\|P_{k_{1}}w\|_{L^{\infty}_{t,x}}2^{k_{2}+k_{3}}\|P_{k_{2}}Q_{\geq k_{2}+k_{3}+40}f\|_{L^{2}_{t,x}}\|P_{k_{3}}g\|_{L^{2,\infty}_{\mathbf{e}_{j}}}\\
&\lesssim \sum_{k_{j}}2^{k_{4}(n-1)/2}2^{k_{1}n/2}\|P_{k_{1}}w\|_{L^{\infty}_{t}L^{2}_{x}}2^{k_{2}+k_{3}}\|P_{k_{2}}f\|_{X^{0,1/2,1}}
2^{-(k_{2}+k_{3})/2}\|P_{k_{3}}g\|_{L^{2,\infty}_{\mathbf{e}_{j}}}\\
&\lesssim \sum_{k_{j}}2^{(k_{4}-k_{2})(n-1)/2}\|w\|_{Y^{n/2}}\|f\|_{F^{n/2}}\|g\|_{F^{n/2}}.
\end{align*}

Noticing that
\begin{align*}
-2\nabla\phi\cdot\nabla\psi=(i\partial_{t}+\Delta)\phi\cdot\psi+\phi\cdot(i\partial_{t}+\Delta)\psi
-(i\partial_{t}+\Delta)(\phi\cdot\psi),
\end{align*}
write $\Theta:=i\partial_{t}+\Delta$ and let
\begin{align*}
\phi=P_{k_{2}}Q_{\leq k_{2}+k_{3}+39}f,\;\; \psi=P_{k_{3}}g.
\end{align*}
Then, the term $H_{12}$ can be easily separated as
\begin{align*}
H_{12}&\leq\sum_{k_{j}}2^{k_{4}n/2}\left\|P_{k_{4}}\left[P_{k_{1}}w Q_{\leq k_{2}+k_{3}}(P_{k_{2}}\Theta Q_{\leq k_{2}+k_{3}+39}f \cdot P_{k_{3}}g)\right]\right\|_{N^{*}_{k_{4}}}\\
&\quad+\sum_{k_{j}}2^{k_{4}n/2}\left\|P_{k_{4}}\left[P_{k_{1}}w Q_{\leq k_{2}+k_{3}}(P_{k_{2}}Q_{\leq k_{2}+k_{3}+39}f \cdot P_{k_{3}}\Theta g)\right]\right\|_{N^{*}_{k_{4}}}\\
&\quad+\sum_{k_{j}}2^{k_{4}n/2}\left\|P_{k_{4}}\left[P_{k_{1}}w Q_{\leq k_{2}+k_{3}}\Theta(P_{k_{2}}Q_{\leq k_{2}+k_{3}+39}f \cdot P_{k_{3}} g)\right]\right\|_{N^{*}_{k_{4}}}\\
&=:H_{121}+H_{122}+H_{123}.
\end{align*}

For the term $H_{121}$, using the fact that the norm on $N^{*}_{k_{3}}$ is $L_{t}^{1}L_{x}^{2}$, Bernstein's inequality and the Plancherel equality  yields that
\begin{align*}
H_{121}&\lesssim \sum_{k_{j}}2^{k_{4}n/2}2^{k_{1}n/2}\|P_{k_{1}}w\|_{L_{t}^{\infty}L_{x}^{2}}\|P_{k_{2}}\Theta Q_{\leq k_{2}+k_{3}+39}f\|_{L_{t,x}^{2}}\|P_{k_{3}}g\|_{L_{t}^{2}L_{x}^{\infty}}.
\end{align*}
By the fact that $P_{k_{2}}\Theta Q_{\leq k_{2}+k_{3}+39}=Q_{\leq k_{2}+k_{3}+39}\Theta P_{k_{2}}$ and Lemma \ref{Qjk-bdd}, we have
\begin{align*}
H_{121}&\lesssim \sum_{k_{j}}2^{k_{4}n/2}2^{k_{1}n/2}\|P_{k_{1}}w\|_{L_{t}^{\infty}L_{x}^{2}}2^{k_{2}}\|P_{k_{2}}f\|_{Z_{k_{2}}}
\|P_{k_{3}}g\|_{L_{t}^{2}L_{x}^{\frac{2n}{n-2}}} \\
&\lesssim \sum_{k_{j}}2^{(k_{4}-k_{3})n/2+(k_{2}-k_{3})}\|w\|_{Y^{n/2}}\|f\|_{Z^{n/2}}\|g\|_{F^{n/2}}.
\end{align*}

For the term $H_{122}$, by the fact that the norm on $N^{*}_{k_{3}}$ is $L^{1,2}_{\mathbf{e}_{j}}$ and applying Minkowski's inequality and Bernstein's inequality, we conclude that
\begin{align*}
H_{122}
&\lesssim \sum_{k_{j}}2^{k_{4}(n-1)/2}\|P_{k_{1}}w\|_{L^{\infty}_{t,x}}\|P_{k_{2}}Q_{\leq k_{2}+k_{3}+39}f\|_{L^{2,\infty}_{\mathbf{e}_{j}}} \|P_{k_{3}}\Theta g\|_{L^{2}_{t,x}} \\
&\lesssim \sum_{k_{j}}2^{k_{4}(n-1)/2}2^{k_{1}n/2}\|P_{k_{1}}w\|_{L^{\infty}_{t}L^{2}_{x}}\|P_{k_{2}}f\|_{L^{2,\infty}_{\mathbf{e}_{j}}}
2^{k_{2}+k_{3}}\|P_{k_{3}}g\|_{X^{0,\frac{1}{2},\infty}}2^{-\frac{j}{2}} \\
&\lesssim \sum_{k_{j}}2^{(k_{4}-k_{3})(n-2)/2}\|w\|_{Y^{n/2}}\|f\|_{F^{n/2}}\|g\|_{F^{n/2}},
\end{align*}
where we have used convolution properties in the second line, and the fact that the bound of $|\tau+|\xi|^{2}|$ is $ 2^{k_{2}+k_{3}}$ by using the effect of the output modulation operator $Q_{\leq k_{2}+k_{3}}$, i.e.,
\begin{align}\label{4.19}
\begin{split}
Q_{\leq k_{2}+k_{3}}\Theta g
&=\mathscr{F}^{-1}_{\tau,\xi}\left\{\chi_{\leq k_{2}+k_{3}}(\tau+|\xi|^{2})\cdot(-\tau-|\xi|^{2})\mathscr{F}g\right\}\\
&\lesssim 2^{k_{2}+k_{3}}Q_{\leq k_{2}+k_{3}}\cdot g.
\end{split}
\end{align}

For the term $H_{123}$, we divide the modulation operator $Q_{\leq k_{2}+k_{3}}$ into two parts for a more exactly estimate. Hence,
\begin{align*}
H_{123}&\leq \sum_{k_{j}}2^{k_{4}n/2}\left\|P_{k_{4}}\left[P_{k_{1}}w Q_{[k_{1}+k_{4}+100,k_{2}+k_{3}]}\Theta(P_{k_{2}}Q_{\leq k_{2}+k_{3}+39}f \cdot P_{k_{3}} g)\right]\right\|_{N^{*}_{k_{4}}}\\
&\quad+\sum_{k_{j}}2^{k_{4}n/2}\left\|P_{k_{4}}\left[P_{k_{1}}w Q_{\leq k_{1}+k_{4}+99}\Theta(P_{k_{2}}Q_{\leq k_{2}+k_{3}+39}f \cdot P_{k_{3}} g)\right]\right\|_{N^{*}_{k_{4}}}\\
&=:L_{1}+L_{2}.
\end{align*}
By the fact that the norm on $N^{*}_{k_{3}}$ is $L^{1,2}_{\mathbf{e}_{j}}$, using Lemma \ref{Qjk-bdd}, we can bound the term $L_{2}$ by
\begin{align*}
L_{2}
&\lesssim \sum_{k_{j}}2^{k_{4}(n-1)/2}\|P_{k_{1}}w\|_{L^{2,\infty}_{\mathbf{e}_{j}}}
2^{k_{1}+k_{4}}\|P_{k_{2}}Q_{\leq k_{2}+k_{3}+39}f P_{k_{3}} g\|_{L^{2}_{t,x}} \\
&\lesssim \sum_{k_{j}}2^{k_{4}(n-1)/2}\|P_{k_{1}}w\|_{L^{2,\infty}_{\mathbf{e}_{j}}}2^{k_{1}+k_{4}}
\|P_{k_{2}}f\|_{L^{\infty}_{t}L^{2}_{x}}2^{k_{3}(n-2)/2}\|P_{k_{3}}g\|_{L^{2}_{t}L^{\frac{2n}{n-2}}_{x}}\\
&\lesssim \sum_{k_{j}}2^{(k_{4}-k_{3})(n+1)/2+(k_{1}-k_{3})/2}\|w\|_{Y^{n/2}}\|f\|_{F^{n/2}}\|g\|_{F^{n/2}},
\end{align*}
where we have used the fact that $k_{3}\geq k_{1}+21$. For the term $L_{1}$, we use the modulation operator $Q_{k}$ to divide the term $L_{1}$ into two parts. Thus,
\begin{align*}
L_{1}
&\leq \sum_{k_{j}}\sum_{j_{2}=k_{1}+k_{4}+100}^{k_{2}+k_{3}}2^{k_{4}n/2}\left\|P_{k_{4}}Q_{\leq j_{2}-10}\left[P_{k_{1}}w Q_{j_{2}}\Theta(P_{k_{2}}Q_{\leq k_{2}+k_{3}+39}f \cdot P_{k_{3}} g)\right]\right\|_{N^{*}_{k_{4}}}\\
&\quad+\sum_{k_{j}}\sum_{j_{2}=k_{1}+k_{4}+100}^{k_{2}+k_{3}}2^{k_{4}n/2}\left\|P_{k_{4}}Q_{\geq j_{2}-9}\left[P_{k_{1}}w Q_{j_{2}}\Theta(P_{k_{2}}Q_{\leq k_{2}+k_{3}+39}f \cdot P_{k_{3}} g)\right]\right\|_{N^{*}_{k_{4}}}\\
&=:L_{11}+L_{12}.
\end{align*}

Noted that the norm of the term $L_{11}$ on $N^{*}_{k_{3}}$ is $L^{1}_{t}L^{2}_{x}$. Applying
Minkowski's inequality, Bernstein's inequality and employing the method of \eqref{4.19}, we deduce that{\small
\begin{align*}
L_{11}
&\lesssim \sum_{k_{j}}\sum_{j_{2}=k_{1}+k_{4}+100}^{k_{2}+k_{3}}2^{k_{4}n/2}\|Q_{j_{3}}P_{k_{1}}w\|_{L^{2}_{t}L^{\infty}_{x}}
\|Q_{j_{2}}\Theta(P_{k_{2}}Q_{\leq k_{2}+k_{3}+39}f \cdot P_{k_{3}} g)\|_{L^{2}_{t,x}} \\
&\lesssim \sum_{k_{j}}\sum_{j_{2}=k_{1}+k_{4}+100}^{k_{2}+k_{3}}2^{k_{4}n/2}2^{k_{1}n/2}\sum_{j_{3}\leq j_{2}-10}2^{j_{3}}\|Q_{j_{3}}P_{k_{1}}w\|_{L^{2}_{t,x}}2^{-j_{3}}2^{j_{2}}\|P_{k_{2}}f\|_{L^{\infty}_{t}L^{2}_{x}}
\|P_{k_{3}} g\|_{L^{2}_{t}L^{\infty}_{x}} \\
&\lesssim \sum_{k_{j}}2^{k_{4}n/2}2^{k_{1}n/2}\left[2^{-(k_{1}+k_{4})}-2^{-(k_{2}+k_{3})}\right]\cdot
\left[2^{(k_{1}+k_{4})}-2^{(k_{2}+k_{3})}\right]2^{k_{3}(n-2)/2}\\
&\quad\|P_{k_{1}}w\|_{X^{0,1,2}}\|P_{k_{2}}f\|_{L^{\infty}_{t}L^{2}_{x}}\|P_{k_{3}} g\|_{L^{2}_{t}L^{\frac{2n}{n-2}}_{x}} \\
&\lesssim \sum_{k_{j}}\left[2^{(k_{4}-k_{2})(n-2)/2}+2^{(k_{4}-k_{2})(n+2)/2+2(k_{1}-k_{3})}\right]
\|w\|_{Y^{n/2}}\|f\|_{F^{n/2}}\|g\|_{F^{n/2}},
\end{align*}}
where the projector $Q_{j_{3}}$ represents the influence (about frequency) of on $P_{k_{1}}w$ by $Q_{\leq j_{2}-10}$.

For the term $L_{12}$, using the fact that the norm on $N^{*}_{k_{3}}$ is $X^{0,-\frac{1}{2},1}_{+}$ and Lemma \ref{Qjk-bdd}, we obtain {\small
\begin{align*}
L_{12}
&\lesssim \sum_{k_{j}}\sum_{j_{2}=k_{1}+k_{4}+100}^{k_{2}+k_{3}}2^{k_{4}n/2}\sum_{j_{3}\geq j_{2}-9}2^{-\frac{j_{3}}{2}}
\|P_{k_{1}}w\|_{L^{\infty}_{t,x}}\|Q_{j_{2}}\Theta(P_{k_{2}}Q_{\leq k_{2}+k_{3}+39}f P_{k_{3}} g)\|_{L^{2}_{t,x}}\\
&\lesssim \sum_{k_{j}}\sum_{j_{2}=k_{1}+k_{4}+100}^{k_{2}+k_{3}}2^{k_{4}n/2}2^{\frac{j_{2}}{2}}2^{k_{1}n/2}
\|P_{k_{1}}w\|_{L^{\infty}_{t}L^{2}_{x}}\|P_{k_{2}}f\|_{L^{\infty}_{t}L^{2}_{x}}
\|P_{k_{3}}g\|_{L^{2}_{t}L^{\infty}_{x}}\\
&\lesssim \sum_{k_{j}}2^{k_{4}n/2}2^{k_{1}n/2}2^{k_{3}(n-2)/2}\left[2^{(k_{1}+k_{4})/2}-2^{(k_{2}+k_{3})/2}\right]
\|P_{k_{1}}w\|_{L^{\infty}_{t}L^{2}_{x}}\|P_{k_{2}}f\|_{L^{\infty}_{t}L^{2}_{x}}\|P_{k_{3}}g\|_{L^{2}_{t}L^{\frac{2n}{n-2}}_{x}}\\
&\lesssim  \sum_{k_{j}}2^{(k_{4}-k_{2})n/2}(1+2^{k_{4}-k_{3}})\|w\|_{Y^{n/2}}\|f\|_{F^{n/2}}\|g\|_{F^{n/2}},
\end{align*}}
\!\!where we have used the fact that $2^{(k_{1}+k_{4})/2}-2^{(k_{2}+k_{3})/2}\leq 2^{k_{4}}-2^{k_{3}}$ in the last line. Hence, we can get $I\leq C \|w\|_{Y^{n/2}}\|f\|_{F^{n/2}}\|g\|_{F^{n/2}}$. Summing up all the estimates, we obtain \eqref{4.2}.
\end{proof}

\begin{proposition}[Noninear estimate]\label{pro nonlinear-es}
Assume that $n\geq3$, $u(t,x)$ is the solution of \eqref{deri-Ginz} satisfying $\|u\|_{F^{n/2}\cap Z^{n/2}}\ll1$, $J(u(t,x))$ is the nonlinear term, for any $\varepsilon>0$,
\begin{align*}
u_{t}-(\varepsilon+ai)\Delta u=J(u(t,x)),\quad u(0,x)=u_{0}.
\end{align*}
Then, we have
\begin{align}
\begin{split}
\|J(u(x,t))\|_{N^{n/2}}
&\leq C\bigg\{\|v\|_{L^{\infty}_{t,x}}\|u\|_{F^{n/2}\cap Z^{n/2}}+\frac{1}{1-\|u\|^{2}_{F^{n/2}\cap Z^{n/2}}}\\
&\quad\left[\|u\|^{3}_{F^{n/2}\cap Z^{n/2}}+\|v\|_{L^{\infty}_{t,x}}\|u\|_{F^{n/2}\cap Z^{n/2}}
+\|v\|_{L^{\infty}_{t,x}}\|u\|^{3}_{F^{n/2}\cap Z^{n/2}}\right] \\
&\quad+\frac{1}{(1-\|u\|^{2}_{F^{n/2}\cap Z^{n/2}})^{2}}\|v\|_{L^{\infty}_{t,x}}
\|u\|^{3}_{F^{n/2}\cap Z^{n/2}}\bigg\},\label{4.17}
\end{split}
\end{align}
where the constant $C>0$ is independent of $\varepsilon$.
\end{proposition}

\begin{proof}
By Taylor's expansion and the fact that $F^{n/2}\cap Z^{n/2}\subset Y^{n/2}\subset L^{\infty}_{t,x}$, we have
\begin{align*}
\left\|\frac{1}{1+|u|^{2}}\right\|_{L^{\infty}_{t,x}}
=\left\|\sum_{k=0}^{n}(-1)^{k}|u|^{2k}\right\|_{L^{\infty}_{t,x}}
\lesssim \left\|\sum_{k=0}^{n}(-1)^{k}|u|^{2k}\right\|_{F^{n/2}\cap Z^{n/2}}
\lesssim \frac{1}{1-\|u\|^{2}_{F^{n/2}\cap Z^{n/2}}}
\end{align*}
Noted that $J(u(x,t))=J_{1}(u(t,x))+J_{2}(u(t,x))+J_{3}(u(t,x))$. For the term $J_{1}(u(t,x))$, combining Lemma \ref{Lemma4.3} and Lemma \ref{Lemma4.4}, we find that
\begin{align}
\begin{split}
\|J_{1}(u(t,x))\|_{N^{n/2}}&=\|-(1+i)(v\cdot\nabla)u\|_{N^{n/2}}
\lesssim \|v\|_{L^{\infty}_{t,x}}\left\|\frac{u^{2}\nabla u}{u^{2}}\right\|_{N^{n/2}} \\
&\lesssim \frac{\|v\|_{L^{\infty}_{t,x}}}{\|u\|^{2}_{F^{n/2}\cap Z^{n/2}}}\|u^{2}\nabla u\|_{N^{n/2}} \\
&\lesssim \|v\|_{L^{\infty}_{t,x}}\|u\|_{F^{n/2}\cap Z^{n/2}},\label{4.14}
\end{split}
\end{align}
For the term $J_{2}(u(t,x))$, by using Lemma \ref{Lemma4.1} and Lemma \ref{Lemma4.2}, one has
\begin{align}
\begin{split}
\|J_{2}(u(t,x))\|_{N^{n/2}}
&\lesssim \frac{1}{1-\|u\|^{2}_{F^{n/2}\cap Z^{n/2}}}\|\bar{u}(\nabla u)^{2}\|_{N^{n/2}}\\
&\lesssim \frac{1}{1-\|u\|^{2}_{F^{n/2}\cap Z^{n/2}}}\|u\|_{Y^{n/2}}\|u\|^{2}_{F^{n/2}\cap Z^{n/2}}\\
&\lesssim \frac{1}{1-\|u\|^{2}_{F^{n/2}\cap Z^{n/2}}}\|u\|^{3}_{F^{n/2}\cap Z^{n/2}},\label{4.13}
\end{split}
\end{align}
For the term $J_{3}(u(t,x))$, which mainly consists of $F(u,\bar{u})$ and $H(u,\bar{u})$, employing Lemma \ref{Lemma4.3} and Lemma \ref{Lemma4.4}, we obtain
\begin{align}
\begin{split}
\left\|\frac{{\rm{Im}}F}{1+|u|^{2}}\right\|_{N^{n/2}}
&\lesssim\left\|\frac{F}{(1+|u|^{2})^2}\right\|_{N^{n/2}} \\
&\lesssim \left\|\frac{(1+|\bar{u}|^{2})|(1-|u|^{2})(v\cdot\nabla)u|}{(1+|u|^{2})^2}\right\|_{N^{n/2}} \\
&\lesssim \frac{1}{1-\|u\|^{2}_{F^{n/2}\cap Z^{n/2}}}\left\|[(v\cdot\nabla)u+|u|^{2}(v\cdot\nabla)u]\right\|_{N^{n/2}} \\
&\lesssim \frac{1}{1-\|u\|^{2}_{F^{n/2}\cap Z^{n/2}}}\|v\|_{L^{\infty}_{t,x}}
\left(\|u\|_{F^{n/2}\cap Z^{n/2}}+\|u\|^{3}_{F^{n/2}\cap Z^{n/2}}\right),\label{4.15}
\end{split}
\end{align}
and
\begin{align}
\begin{split}
\left\|\frac{{\rm{Re}}H}{1+|u|^{2}}\right\|_{N^{n/2}}
&\lesssim\left\|\frac{H}{(1+|u|^{2})}\right\|_{N^{n/2}} \\
&\lesssim \left\|\frac{4(v\cdot\nabla)u(\bar{u}^{2}+|u|^{2})}{(1+|u|^{2})^2}\right\|_{N^{n/2}} \\
&\lesssim \frac{1}{(1-\|u\|^{2}_{F^{n/2}\cap Z^{n/2}})^{2}}\|v\|_{L^{\infty}_{t,x}}\left\||u|^{2}\nabla u\right\|_{N^{n/2}} \\
&\lesssim \frac{1}{(1-\|u\|^{2}_{F^{n/2}\cap Z^{n/2}})^{2}}\|v\|_{L^{\infty}_{t,x}}\|u\|^{3}_{F^{n/2}\cap Z^{n/2}},\label{4.16}
\end{split}
\end{align}
Combining the estimates \eqref{4.14}--\eqref{4.16}, we get the desired result \eqref{4.17}.
\end{proof}

\section{The existence and uniqueness of strong solutions}\label{Sec5}
In this section, gathering the arguments of Section 3 and Section 4, we prove the well-posedness of the Ginzburg-Landau equation
\eqref{deri-Ginz} by the contraction mapping theorem, and then we prove the main theorem.

\begin{theorem}\label{th exist}
Assume that $n\geq3$, $0<\varepsilon< 1$, the initial data satisfies $\|u_{0}\|_{\dot{B}^{\frac{n}{2}}_{2,1}}\leq \eta$ and $\eta>0$ is small enough. For any $\|v\|_{L^{\infty}_{t,x}}\leq \eta$, then the equation
\eqref{deri-Ginz} has a unique global solution $u(t,x)$ satisfying
\begin{align*}
\|u\|_{F^{n/2}\cap Z^{n/2}}\leq C\eta,
\end{align*}
where the constant $C>0$ is independent of $\varepsilon$.
\end{theorem}

\begin{proof}
We construct the solution map by Duhamel's principle:
\begin{align*}
\Psi_{u_{0}}(u):=e^{(\varepsilon+i)t\Delta}u_{0}+\int_{0}^{t}e^{(\varepsilon+i)(t-s)\Delta}J(u(s))ds,
\end{align*}
where the initial data $u_{0}$ is given.

First, we prove the map $\Psi_{u_{0}}$ is the map from itself to itself, i.e.,
\begin{align*}
\Psi:~F^{n/2}\cap Z^{n/2}\mapsto F^{n/2}\cap Z^{n/2},
\end{align*}
The result follows easily from Proposition \ref{pro linear-es} and Proposition \ref{pro nonlinear-es}.

Next, we show that the map $\Psi_{u_{0}}$ is contractive. Suppose that $u_{1}$ and $u_{2}$ are two solutions of \eqref{deri-Ginz}, and they correspond to initial data $u_{0}$. By using Proposition \ref{pro linear-es}, we get
\begin{align*}
\|\Psi_{u_{0}}(u_{1})-\Psi_{u_{0}}(u_{2})\|_{F^{n/2}\cap Z^{n/2}}
\lesssim \|J(u_{1})-J(u_{2})\|_{N^{n/2}}.
\end{align*}
Notice that the term $J(u(t,x))$ consists of three parts, we take the term $J_{2}(u(t,x))$ as an example to prove the contraction,
\begin{align*}
\|J_{2}(u_{1})-J_{2}(u_{2})\|_{N^{n/2}}
&\lesssim \left\|\frac{\bar{u}_{1}(\nabla u_{1})^{2}}{1+|u_{1}|^{2}}
-\frac{\bar{u}_{2}(\nabla u_{2})^{2}}{1+|u_{2}|^{2}} \right\|_{N^{n/2}}\\
&\lesssim \sup_{u\in\{{u_{1},u_{2}}\}}\frac{1}{1-\|u\|^{2}_{F^{n/2}\cap Z^{n/2}}}\|(\bar{u}_{1}-\bar{u}_{2})(\nabla u)^{2}\|_{N^{n/2}}\\
&\lesssim \sup_{u\in\{{u_{1},u_{2}}\}}\frac{\|u\|^{2}_{F^{n/2}\cap Z^{n/2}}}{1-\|u\|^{2}_{F^{n/2}\cap Z^{n/2}}}
\|u_{1}-u_{2}\|_{F^{n/2}\cap Z^{n/2}}.
\end{align*}
Similarly, we deal with the terms $J_{1}(u(t,x))$ and $J_{3}(u(t,x))$ as
\begin{align*}
\|J_{1}(u_{1})-J_{1}(u_{2})\|_{N^{n/2}}
&\lesssim  \|v\|_{L^{\infty}_{t,x}}\|u_{1}-u_{2}\|_{F^{n/2}\cap Z^{n/2}},
\end{align*}
and
\begin{align*}
\|J_{3}(u_{1})-J_{3}(u_{2})\|_{N^{n/2}}
&\lesssim\bigg\{\frac{1}{1-\|u\|^{2}_{F^{n/2}\cap Z^{n/2}}}\left[\|v\|_{L^{\infty}_{t,x}}
+\|v\|_{L^{\infty}_{t,x}}\|u\|^{2}_{F^{n/2}\cap Z^{n/2}}\right]\\
&\quad+\frac{1}{(1-\|u\|^{2}_{F^{n/2}\cap Z^{n/2}})^{2}}\|v\|_{L^{\infty}_{t,x}}
\|u\|^{2}_{F^{n/2}\cap Z^{n/2}}\bigg\}\|u_{1}-u_{2}\|_{F^{n/2}\cap Z^{n/2}}.
\end{align*}
Combining the above estimates with Proposition \ref{pro nonlinear-es}, one has
\begin{align}\label{5.2}
\begin{split}
&\|\Psi_{u_{0}}(u_{1})-\Psi_{u_{0}}(u_{2})\|_{F^{n/2}\cap Z^{n/2}}\\
&\leq C_{1}\sum_{i=1}^{3}\|J_{i}(u_{1})-J_{i}(u_{2})\|_{N^{n/2}}\\
&\leq C_{1}\bigg\{\|v\|_{L^{\infty}_{t,x}}+\frac{1}{1-\|u\|^{2}_{F^{n/2}\cap Z^{n/2}}}\left[\|u\|^{2}_{F^{n/2}\cap Z^{n/2}}+\|v\|_{L^{\infty}_{t,x}}
+\|v\|_{L^{\infty}_{t,x}}\|u\|^{2}_{F^{n/2}\cap Z^{n/2}}\right] \\
&\quad+\frac{1}{(1-\|u\|^{2}_{F^{n/2}\cap Z^{n/2}})^{2}}\|v\|_{L^{\infty}_{t,x}}
\|u\|^{2}_{F^{n/2}\cap Z^{n/2}}\bigg\}\|u_{1}-u_{2}\|_{F^{n/2}\cap Z^{n/2}}.
\end{split}
\end{align}
Applying the relationship between $\|u(t,x)\|_{F^{n/2}\cap Z^{n/2}}$ and $\|u_{0}\|_{\dot{B}^{\frac{n}{2}}_{2,1}}$ (see Proposition \ref{pro linear-es}), we get \small{
\begin{align}\label{5.1}
\begin{split}
&\|u(t,x)\|_{F^{\frac{n}{2}}\cap Z^{\frac{n}{2}}}\\
&\leq C_{2}\|u_{0}\|_{\dot{B}^{\frac{n}{2}}_{2,1}}+C_{1}\|J(x,t)\|_{N^{\frac{n}{2}}}\\
&\leq C_{2}\|u_{0}\|_{\dot{B}^{\frac{n}{2}}_{2,1}}+C_{1}\bigg\{\|v\|_{L^{\infty}_{t,x}}+\frac{1}{1-\|u\|^{2}_{F^{n/2}\cap Z^{n/2}}}
\left[\|u\|^{2}_{F^{n/2}\cap Z^{n/2}}+\|v\|_{L^{\infty}_{t,x}}
+\|v\|_{L^{\infty}_{t,x}}\|u\|^{2}_{F^{n/2}\cap Z^{n/2}}\right]\\
&\quad+\frac{1}{(1-\|u\|^{2}_{F^{n/2}\cap Z^{n/2}})^{2}}\|v\|_{L^{\infty}_{t,x}}
\|u\|^{2}_{F^{n/2}\cap Z^{n/2}}\bigg\}\|u\|_{F^{n/2}\cap Z^{n/2}}\\
&=:C_{2}\|u_{0}\|_{\dot{B}^{\frac{n}{2}}_{2,1}}+C_{1}\Gamma\big(\|v\|_{L^{\infty}_{t,x}},\|u\|_{F^{n/2}\cap Z^{n/2}}\big)\|u\|_{F^{n/2}\cap Z^{n/2}}.
\end{split}
\end{align}}
Next, if we suppose that $\|u\|_{F^{n/2}\cap Z^{n/2}}\leq C_{3}\|u_{0}\|_{\dot{B}^{\frac{n}{2}}_{2,1}},~\|v\|_{L^{\infty}_{t,x}}\leq \sigma_{1}$, and $C_{3}$ satisfies $C_{2}\leq \frac{C_{3}}{4}$, we can easily obtain
\begin{align*}
\Gamma\big(\|v\|_{L^{\infty}_{t,x}},\|u\|_{F^{n/2}\cap Z^{n/2}}\big)
\leq \sigma_{1}+\frac{\sigma_{1}+C_{3}^{2}\|u_{0}\|_{\dot{B}^{\frac{n}{2}}_{2,1}}^{2}+\sigma_{1} C_{3}^{2}\|u_{0}\|_{\dot{B}^{\frac{n}{2}}_{2,1}}^{2}}
{1-C_{3}^{2}\|u_{0}\|_{\dot{B}^{\frac{n}{2}}_{2,1}}^{2}}
+\frac{\sigma_{1} C_{3}^{2}\|u_{0}\|_{\dot{B}^{\frac{n}{2}}_{2,1}}^{2}}{\big(1-C_{3}^{2}\|u_{0}\|_{\dot{B}^{\frac{n}{2}}_{2,1}}^{2}\big)^{2}}
=:C_{4},
\end{align*}
 assuming that $\|u_{0}\|_{\dot{B}^{\frac{n}{2}}_{2,1}}\leq \sigma_{2}$ and $\eta:=\min{\{\sigma_{1},\sigma_{2}\}}$ is small enough, yields that $C_{4}$ is small and $C_{1}C_{4}\leq \frac{1}{4}$, then \eqref{5.1} can obtain
\begin{align*}
\|u(t,x)\|_{F^{\frac{n}{2}}\cap Z^{\frac{n}{2}}}
\leq (C_{2}+C_{1}C_{4}C_{3})\|u_{0}\|_{\dot{B}^{\frac{n}{2}}_{2,1}}
\leq \frac{C_{3}}{2}\|u_{0}\|_{\dot{B}^{\frac{n}{2}}_{2,1}}.
\end{align*}
Then applying the bootstrap argument which implies the following relationship indeed hold:
\begin{align*}
\|u\|_{F^{n/2}\cap Z^{n/2}}\leq C_{3}\|u_{0}\|_{\dot{B}^{\frac{n}{2}}_{2,1}}
\end{align*}

Finally gathering the above arguments, when $\|v\|_{L^{\infty}_{t,x}},~\|u_{0}\|_{\dot{B}^{\frac{n}{2}}_{2,1}}\leq \eta$ is small enough, inequality \eqref{5.2} satisfies
\begin{align}
\|\Psi_{u_{0}}(u_{1})-\Psi_{u_{0}}(u_{2})\|_{F^{n/2}\cap Z^{n/2}}
\leq \frac{1}{4}\|u_{1}-u_{2}\|_{F^{n/2}\cap Z^{n/2}}.\label{contractive map}
\end{align}
By the contraction mapping theorem, the global well-posedness can be proved. Hence, we complete the proof of Theorem \ref{th exist}.
\end{proof}

\underline{\bf Proof of Theorem \ref{main-th}}: By the invertibility of the stereographic projection transform, the proof of the well-posedness for the strong solution of equation \eqref{main-eq} is equivalent to the proof of the well-posedness of equation \eqref{deri-Ginz}. Theorem \ref{main-th} follows from Theorem \ref{th exist} by the above arguments. Then, we complete the proof of Theorem \ref{main-th}.

\smallskip

\section*{Acknowledgments}
H. Wang's research is supported by the National Natural Science Foundation of China (Grant No.~11901066), the
Natural Science Foundation of Chongqing (No.~cstc2019jcyj-msxmX0167) and projects No.~2022CDJXY-001,  No.~2020CDJQY-A040 supported by the Fundamental Research Funds for the Central Universities.

\end{document}